\documentclass{amsart}
\usepackage{graphicx}
\usepackage[all]{xy}
\usepackage{amscd}
\usepackage{amssymb}
\usepackage{amsmath}
\usepackage{amsthm}
\usepackage{amsfonts}
\usepackage{graphics}
\usepackage{epsfig,psfrag}
\usepackage[english]{babel}
\usepackage{latexsym}
\usepackage{mathrsfs}

\newtheorem{theorem}{Theorem}[section]

\newtheorem{lemma}[theorem]{Lemma}

\newtheorem{proposition}[theorem]{Proposition}
\newtheorem{question}[theorem]{Question}

\theoremstyle{definition}

\newtheorem{definition}[theorem]{Definition}

\theoremstyle{remark}
\newtheorem{remark}[theorem]{Remark}

\begin{document}
\title[Equivariant embeddings of $\Sigma_g$ into $\mathbb{R}^n$ with minimal dimensions]
{Equivariant embeddings of Riemann surfaces in Euclidean spaces with minimal dimensions}

\author{Chao Wang}
\address{School of Mathematical Sciences, Key Laboratory of MEA(Ministry of Education) \& Shanghai Key Laboratory of PMMP, East China Normal University, Shanghai 200241, China}
\email{chao\_{}wang\_{}1987@126.com}

\author{Zhongzi Wang}
\address{School of Mathematical Sciences, Peking University, Beijing 100871, China}
\email{wangzz22@stu.pku.edu.cn}



\thanks{The first author is supported by National Natural Science Foundation of China (NSFC), grant Nos. 12131009 and 12371067, and Science and Technology Commission of Shanghai Municipality (STCSM), grant No. 22DZ2229014.}

\begin{abstract}
Let $\Sigma_g$ be a closed Riemann surface of genus $g$. Let $G$ be a finite subgroup of the automorphism group of $\Sigma_g$. It is well known that there exists a smooth $G$-equivariant embedding from $\Sigma_g$ to some Euclidean space $\mathbb{R}^n$. Let $d_g(G)$ be the minimal possible $n$ for $(\Sigma_g,G)$. We compute the value of $d_g(G)$ in certain cases. Especially, we show that: for the automorphism group of the closed Riemann surface which comes from the principal congruence subgroup of level $p$, where $p\geq 7$ is prime, $d_g(G)=p+1$. As a corollary, the minimal $n$ for the Hurwitz action on the Klein quartic is equal to $8$.

Three kinds of methods are used in the computation, which are related to the representations of groups, the equivariant triangulations, and the orbifold theory, respectively. The methods are also used to provide two kinds of upper bounds: $d_g(G)\leq |G|$ if $|G|\geq 5$; and $d_g(G)\leq 12(g-1)$ if $g\geq 2$.
\end{abstract}

\date{}
\maketitle

\tableofcontents


\newpage

\section{Introduction}\label{sec:intro}
In differential topology, a famous theorem of Whitney (see \cite{Wh}) says that any smooth $m$-dimensional manifold can be embedded into $\mathbb{R}^{2m}$. In general, the value $2m$ cannot be improved. For example, a circle cannot lie in $\mathbb{R}^1$, and a given closed nonorientable surface cannot lie in $\mathbb{R}^3$. On the other hand, for some special classes of manifolds, it is possible to make the value smaller. For example, it is known that any smooth $3$-dimensional manifold can be embedded into $\mathbb{R}^5$ (see \cite{Hi,Ro,Wal}). It is in general quite hard to determine the minimal value for specific manifolds so that the embedding exists, and there are rich results in this field.

The Whitney embedding theorem has an equivariant version, which says that if there is a compact Lie group $G$ acting on a smooth manifold $M$, with only finitely many orbit types, then there exists a smooth $G$-equivariant embedding from $M$ to some Euclidean space $\mathbb{R}^n$, where $G$ acts orthogonally on $\mathbb{R}^n$ (see \cite{Mo,Pa}). Notice that when $G$ is finite or $M$ is compact, there are finitely many orbit types, and so the conclusion holds. This equivariant embedding theorem contains no estimation of the value of $n$, but if $G$ is finite and $M$ can be embedded into $\mathbb{R}^k$, then one can pick $n=k|G|$ (see \cite{Wan2}). It is natural to ask what is the minimal value of $n$ for finite $G$-actions on $M$ so that the equivariant embedding exists. In this paper, we will study this question in the first nontrivial case, where $M$ is a connected closed orientable surface and $G$ preserves the orientations on $M$.

Let $\mathrm{O}(n)$ denote the orthogonal group acting on $\mathbb{R}^n$. For our purposes, we state the definition of an equivariant embedding and our main question precisely below. For simplicity, a pair $(M,G)$ is used to denote a $G$-action on a connected compact manifold $M$, which is always required to be faithful.

\begin{definition}\label{def:EquiEmb}
Given $(M,G)$, an embedding $e:M\rightarrow\mathbb{R}^n$ is called {\it $G$-equivariant} if there exists an embedding $\rho: G\rightarrow\mathrm{O}(n)$ such that $e\circ h(x)=\rho(h)\circ e(x)$ for each element $h\in G$ and each point $x\in M$.
\end{definition}

We emphasize that the $G$-action and the embedding $e$ need not be smooth. Let $\Sigma_g$ denote the connected closed orientable surface of genus $g$. It is well known that for such a pair $(\Sigma_g,G)$, there exists a $G$-invariant complex structure on $\Sigma_g$, so that $G$ acts on $\Sigma_g$ by biholomorphic homeomorphisms, and if $g\geq 2$, then there exists a $G$-invariant hyperbolic structure on $\Sigma_g$, so that $G$ acts on $\Sigma_g$ by isometries. Hence we can view $G$ as a subgroup of the automorphism group of a Riemann surface, or the orientation-preserving isometry group of a hyperbolic surface if $g\geq 2$. We will mainly consider the following question.

\begin{question}\label{que:main}
The minimal embedding dimension for $(\Sigma_g,G)$ is defined to be the minimal $n$ so that there exists a smooth $G$-equivariant embedding $e:\Sigma_g\rightarrow\mathbb{R}^n$, and we denote it by $d_g(G)$. Then, what is the exact value of $d_g(G)$?
\end{question}

In \cite{Rü2}, the cyclic $G$-actions on Riemann surfaces $\Sigma_g$ satisfying $d_g(G)=3$ are classified. And it is easy to find other examples of $(\Sigma_g,G)$ so that $d_g(G)=3$. But in general we have not found many results about $d_g(G)$ in the literature yet. First we compute the value of $d_g(G)$ in certain cases and get the following result.

\begin{proposition}\label{thm:Cyclic}
For each integer $n\geq 3$, there exist infinitely many $(\Sigma_g,G)$ such that $d_g(G)=n$, where $G$ is cyclic, and it has order $n$ if $n$ is prime.
\end{proposition}

The computation usually consists of two parts: on one hand, we get restrictions on the required representation $\rho$, and obtain a lower bound of $d_g(G)$; on the other hand, we find a suitable representation $\rho$, and construct an equivariant embedding realizing the lower bound. So we have certain conditions on the representations to detect whether there exists a smooth $G$-equivariant embedding. By applying these conditions to the regular representation of $G$, we have the result below.

\begin{theorem}\label{thm:GEBound}
Let $\Sigma_g$ be a closed Riemann surface of genus $g$. Let $G$ be a finite subgroup of the automorphism group of $\Sigma_g$. Then $d_g(G)\leq |G|$ if $|G|\geq 5$.
\end{theorem}

This condition $|G|\geq 5$ is necessary. Proposition~\ref{thm:Cyclic} implies that the equality in Theorem~\ref{thm:GEBound} can hold for certain $G$-actions of prime orders. At present, we do not know exactly when the equality holds. On the other hand, a very famous theorem of Hurwitz says that $|G|\leq 84(g-1)$ if $g\geq 2$ (see \cite{Hu}). Similarly, it is not known exactly when the equality holds. Now Hurwitz's theorem provides an upper bound of $d_g(G)$ in terms of $g$. However, it is far from optimal.

\begin{theorem}\label{thm:SLBound}
Let $\Sigma_g$ be a closed Riemann surface with genus $g\geq 2$. Let $G$ be a subgroup of the automorphism group of $\Sigma_g$. Then $d_g(G)\leq 12(g-1)$.
\end{theorem}

According to Theorem~\ref{thm:GEBound}, we need to deal with the case when $|G|>12(g-1)$, which is naturally related to hyperbolic triangle groups and equivariant triangulations. So we have certain techniques to construct a smooth equivariant embedding $\Sigma_g\rightarrow\mathbb{R}^n$, where $n$ is relatively small, from an equivariant triangulation of $\Sigma_g$. We believe that the bound in Theorem~\ref{thm:SLBound} can be improved further, but this will need a deeper understanding of the Hurwitz groups, the ones having order $84(g-1)$. It is actually a main motivation of this study to compute $d_g(G)$ for the first Hurwitz group, which acts on the Klein quartic (see the review MR4419628 of \cite{Wan1} for a similar question on embeddings in $S^n$ raised by B. Zimmermann). The remarkable surface was discovered by Klein when he studied those natural actions of principal congruence subgroups on the upper half-plane. One can find descriptions about it from various viewpoints in the book \cite{Le}, and there are also plentiful results about the Hurwitz groups in the references therein.

Let $N>1$ be an integer. Then the principal congruence subgroup of level $N$ in $\mathrm{PSL}(2,\mathbb{Z})$, denoted by $\overline{\Gamma}(N)$, is defined as the image of $\Gamma(N)$ in $\mathrm{PSL}(2,\mathbb{Z})$, where
\[
\Gamma(N)=\Big\{
\begin{pmatrix}
a & b \\
c & d \\
\end{pmatrix}\in\mathrm{SL}(2,\mathbb{Z}) \,\,\Big|\,
\begin{pmatrix}
a & b \\
c & d \\
\end{pmatrix}\equiv
\begin{pmatrix}
1 & 0 \\
0 & 1 \\
\end{pmatrix}\pmod{N}\Big\}.
\]
This subgroup is normal in $\mathrm{PSL}(2,\mathbb{Z})$. And the quotient $\mathrm{PSL}(2,\mathbb{Z})/\overline{\Gamma}(N)$, denoted by $\overline{\Gamma}_N$, is called the modulary group of level $N$. See \cite[Chapter~IV]{Scho}.

Now, $\overline{\Gamma}(N)$ acts freely on the upper half-plane $\mathbb{U}=\{z\in\mathbb{C}\mid\mathrm{Im}(z)>0\}$, and so $\overline{\Gamma}_N$ acts on the Riemann surface $\mathbb{U}/\overline{\Gamma}(N)$. Topologically, $\mathbb{U}/\overline{\Gamma}(N)$ can be obtained from some $\Sigma_g$ by removing finitely many points, and the $\Sigma_g$ can be chosen to be a certain closed Riemann surface $\Sigma(N)$. Moreover, it is known that $\Sigma(N)$ has genus $g\geq 2$ and has its automorphism group isomorphic to $\overline{\Gamma}_N$ if and only if $N\geq 7$. We will only consider the pair $(\Sigma(p),\overline{\Gamma}_p)$ where $p\geq 7$ is a prime number. Note that in this case $\overline{\Gamma}_p\cong\mathrm{PSL}(2,\mathbb{Z}/p\mathbb{Z})$, and it gives the first Hurwitz group when $p=7$. See Section~\ref{subsec:CCofET} for a more precise description of $(\Sigma(p),\overline{\Gamma}_p)$. Combining the techniques from representations and triangulations, we obtain our main result.

\begin{theorem}\label{thm:PCS}
For a prime number $p\geq 7$, the minimal embedding dimension for the pair $(\Sigma(p),\overline{\Gamma}_p)$ is equal to $p+1$. Especially, the minimal embedding dimension for the Hurwitz action on the Klein quartic is equal to $8$.
\end{theorem}

Actually, we obtain more. We have an explicit representation $\rho:\overline{\Gamma}_p\rightarrow\mathrm{O}(p+1)$ and a $\overline{\Gamma}_p$-equivariant embedding $\hat{e}:\Sigma(p)\rightarrow\mathbb{R}^{p+1}$ with respect to $\rho$, where $\hat{e}(\Sigma(p))$ consists of $p(p^2-1)/6$ triangles and can be regarded as an analogue of the surface of a regular icosahedron (corresponding to the case when $p=5$). Then, a required smooth embedding $e$ can be obtained by modifying $\hat{e}(\Sigma(p))$ equivariantly near the vertices and edges of the triangles. So $e(\Sigma(p))$ lies near $\hat{e}(\Sigma(p))$.

For possible further studies, we mention some directions:

The first one is to consider possible variations of $d_g(G)$. As an example, we can remove the condition ``smooth'' on the embedding and define $\hat{d}_g(G)$. In some cases this new minimal embedding dimension will be strictly smaller than $d_g(G)$. Or we can remove the condition ``orthogonal'' on the $G$-action on $\mathbb{R}^n$, just consider those smooth actions on $\mathbb{R}^n$, and define $\tilde{d}_g(G)$. Usually this new invariant will be harder to determine. On the other hand, one can replace $\mathbb{R}^n$ by the $n$-dimensional sphere $S^n$ and define $d^s_g(G)$, $\hat{d}^s_g(G)$, and $\tilde{d}^s_g(G)$ similarly. In this situation, there are more results in the literature. One can see \cite{WW, WWW, WWZZ} for the cases when $n$ is $3$ or $G$ is cyclic. Note that all the variations have $d_g(G)$ as an upper bound.

Another one is to consider possible ``good'' embeddings $e$. For example, one can consider if $e$ can be an equivariant conformal embedding, or can be an equivariant isometric embedding when $g\geq 2$ and $\Sigma_g$ is regarded as a hyperbolic surface. Note that those embeddings in \cite{Rü2} realizing $d_g(G)=3$ are actually conformal. So the result of \cite{Rü2} generalizes the positive solution to the classical problem which asks if every Riemann surface can be conformally embedded into $\mathbb{R}^3$ (see \cite{Ga, Rü1}). It is then natural to ask whether every smooth equivariant embedding of a Riemann surface in $\mathbb{R}^n$ can be replaced by a conformal one. We believe that the equivariant version of the problem also has a positive solution. And note that it is indeed this case when $|G|>12(g-1)>0$, since then the orbifold $\Sigma_g/G$ is a sphere with three singular points and the involved complex structure on $\Sigma_g$ is essentially unique. So parts of our results can be viewed as the generalizations in high dimensions of the result in \cite{Rü2}. Note that the result in \cite{Rü2} is also generalized to the orientation-reversing case in \cite{Co}, where $n$ is $3$ and $G$ is cyclic.

On the other hand, if one replaces $\mathbb{R}^n$ by $S^n$ as above, then we can ask if $e$ can be an equivariant minimal embedding. In history, symmetries have been used in a lot of constructions of complete embedded minimal surfaces (see \cite{KPS, La, Schw}), and recently there are still discoveries in this way (see \cite{BWW}). We also note that the $G$-equivariant conformal minimal immersions of the open Riemann surfaces in $\mathbb{R}^n$ have been deeply studied in \cite{Fo} recently.

\vspace{\fill}

\noindent {\bf Organization.} In Section~\ref{sec:NCandSC}, we give three conditions on the representations of $G$ to detect the existence of smooth $G$-equivariant embeddings, then we use them to prove Proposition~\ref{thm:Cyclic} and Theorem~\ref{thm:GEBound}. In Section~\ref{sec:ETandPT}, we give techniques to obtain smooth equivariant embeddings from equivariant triangulations, then we use them to prove Theorem~\ref{thm:SLBound}. We prove Theorem~\ref{thm:PCS} in Section~\ref{sec:RSfromPCS} and we give an explicit representation to realize $d_g(G)=p+1$ for $(\Sigma(p),\overline{\Gamma}_p)$ in Appendix~\ref{app:RR}.

Note that the two kinds of methods we used to estimate $d_g(G)$ are explained in Sections~\ref{subsec:ECandDI}--\ref{subsec:CC} and Sections~\ref{subsec:ETtoEE}--\ref{subsec:PTof2O}, respectively. The proofs in Sections~\ref{subsec:UB}, \ref{subsec:LAandHTG}, and Section~\ref{sec:RSfromPCS} are their applications, and can be read separately.

As another tool, the theory of orbifolds will be used throughout this paper. See \cite{BMP} for the theory of orbifolds. Also, one will need some basic properties about the representations of finite groups; see \cite{FH}.


\section{Necessary conditions and sufficient conditions}\label{sec:NCandSC}
Given $(\Sigma_g,G)$, in Section~\ref{subsec:ECandDI} and Section~\ref{subsec:CC}, we will first give some conditions on the representation $\rho: G\rightarrow\mathrm{O}(n)$, such that there exists a smooth $G$-equivariant embedding $e:\Sigma_g\rightarrow\mathbb{R}^n$. Then in Section~\ref{subsec:UB} we will use those conditions to prove Proposition~\ref{thm:Cyclic} and Theorem~\ref{thm:GEBound}, and we will give some related results.


\subsection{Dimension inequality and eigenvalue condition}\label{subsec:ECandDI}
For a point $x$ in $\Sigma_g$ or $\mathbb{R}^n$, we will use $\mathrm{Stab}(x)$ to denote its stabilizer in $G$ or $\rho(G)$, respectively. We call $x$ {\it regular} if $\mathrm{Stab}(x)$ is trivial. Otherwise we call $x$ {\it singular}. For an element $h$ in $G$ or $\rho(G)$, we will use $\mathrm{Fix}(h)$ to denote its fixed point set in $\Sigma_g$ or $\mathbb{R}^n$, which will be a finite set or a linear subspace, respectively. Similarly we can define $\mathrm{Fix}(H)$ for a subgroup $H$ of $G$ or $\rho(G)$. Clearly $\mathrm{Fix}(H)=\bigcap_{h\in H}\mathrm{Fix}(h)$.

\begin{definition}
A representation $\rho: G\rightarrow\mathrm{O}(n)$ satisfies the {\it dimension inequality} if for any point $x$ in $\Sigma_g$ and any subgroup $H$ of $G$ so that $\mathrm{Stab}(x)<H$, we have
\[\mathrm{dim}\mathrm{Fix}(\rho(\mathrm{Stab}(x)))>\mathrm{dim}\mathrm{Fix}(\rho(H)),\]
and when $\mathrm{Fix}(G)\subset\Sigma_g$ has at least two points, we have $\mathrm{dim}\mathrm{Fix}(\rho(G))>0$.
\end{definition}

\begin{definition}
A representation $\rho: G\rightarrow\mathrm{O}(n)$ satisfies the {\it eigenvalue condition} if for any singular point $x$ in $\Sigma_g$ and any element $h$ in $\mathrm{Stab}(x)$, all the eigenvalues of the tangent map $dh_x:T_x\Sigma_g\rightarrow T_x\Sigma_g$ are also eigenvalues of $\rho(h)$.
\end{definition}

\begin{lemma}\label{lem:nec}
If there exists a $G$-equivariant embedding $e:\Sigma_g\rightarrow\mathbb{R}^n$, with respect to an embedding $\rho: G\rightarrow\mathrm{O}(n)$, then $\rho$ satisfies the dimension inequality.

If the above $e$ is also smooth, then $\rho$ also satisfies the eigenvalue condition.
\end{lemma}

\begin{proof}
Since $e\circ h(x)=\rho(h)\circ e(x)$ for each element $h\in G$ and each point $x\in\Sigma_g$, we see that $h(x)=x$ if and only if $\rho(h)\circ e(x)=e(x)$. For $x\in\Sigma_g$ and $H\leq G$ such that $\mathrm{Stab}(x)<H$, we pick an element $h\in H\setminus\mathrm{Stab}(x)$. Since $\rho(\mathrm{Stab}(x))<\rho(H)$, we have $\mathrm{Fix}(\rho(\mathrm{Stab}(x)))\supseteq\mathrm{Fix}(\rho(H))$. Now $h(x)\neq x$, so $\rho(h)\circ e(x)\neq e(x)$. Hence $e(x)\in\mathrm{Fix}(\rho(\mathrm{Stab}(x)))\setminus\mathrm{Fix}(\rho(H))$, which implies the required inequality.

If $\mathrm{Fix}(G)\neq\emptyset$, then for $x\in\mathrm{Fix}(G)$, we have $\mathrm{Stab}(x)=G$ and $e(x)\in\mathrm{Fix}(\rho(G))$. Hence $\mathrm{dim}\mathrm{Fix}(\rho(G))>0$ when $\mathrm{Fix}(G)$ has at least two points.

Now assume that $e$ is smooth. Let $x$ be a singular point in $\Sigma_g$, and let $h$ be an element in $\mathrm{Stab}(x)$. Since $e\circ h(x)=\rho(h)\circ e(x)$, we have the relation
\[de_{h(x)}\circ dh_x=d\rho(h)_{e(x)}\circ de_x.\]
Then since $h(x)=x$, we have $de_{h(x)}=de_x$, which identifies $T_x\Sigma_g$ with a subspace of $T_{e(x)}\mathbb{R}^n\cong\mathbb{R}^n$. Also, $\rho(h)\circ e(x)=e(x)$, and $de_x(T_x\Sigma_g)$ is an invariant subspace of $d\rho(h)_{e(x)}$. Since $\rho(h)$ is linear, we can identify $d\rho(h)_{e(x)}$ with $\rho(h)$. Then we see that all the eigenvalues of $dh_x$ are also eigenvalues of $\rho(h)$.
\end{proof}

\begin{remark}
By the proof, we see that Lemma~\ref{lem:nec} actually holds for general pairs $(M,G)$. Also, the eigenvalues should be counted with multiplicity. In our case, the group $\mathrm{Stab}(x)$ is always cyclic, so we only need to consider some generator $h$ of it. Then $dh_x$ has a multiple eigenvalue if and only if the eigenvalue is $-1$.
\end{remark}


\subsection{Codimension condition and existence theorem}\label{subsec:CC}

\begin{definition}
A representation $\rho: G\rightarrow\mathrm{O}(n)$ satisfies the {\it codimension condition} if $n>4$ and $\mathrm{codim}\mathrm{Fix}(\rho(h))>2$ for any nontrivial element $h$ in $G$.
\end{definition}

\begin{theorem}\label{thm:exist}
Given $(\Sigma_g,G)$, if there is a representation $\rho: G\rightarrow\mathrm{O}(n)$ satisfying the codimension condition and dimension inequality, then there is a $G$-equivariant embedding $e:\Sigma_g\rightarrow\mathbb{R}^n$ with respect to $\rho$, where $\rho$ is an embedding.

If $\rho$ also satisfies the eigenvalue condition, then the above $e$ can be smooth.
\end{theorem}

The fact that $\rho$ is an embedding can be derived from the codimension condition or the dimension inequality. So we have $G\cong\rho(G)$. To prove Theorem~\ref{thm:exist}, we will construct an embedding $\overline{e}:\Sigma_g/G\rightarrow\mathbb{R}^n/\rho(G)$ between the orbifolds. The required embedding $e$ will be a lift of $\overline{e}$. We first consider the latter part.

Clearly the regular (resp. singular) points in $\Sigma_g$ map to {\it regular} (resp. {\it singular}) points in $\Sigma_g/G$. Pick a regular point $\star$ in $\Sigma_g$ as the basepoint, and let its image $\overline{\star}$ be the basepoint in $\Sigma_g/G$. By the theory of orbifolds, we have an exact sequence
\[1\rightarrow \pi_1(\Sigma_g,\star)\rightarrow \pi_1(\Sigma_g/G,\overline{\star})\rightarrow G\rightarrow 1,\]
where the map $\phi:\pi_1(\Sigma_g/G,\overline{\star})\rightarrow G$ is defined as follows. For $\alpha\in\pi_1(\Sigma_g/G,\overline{\star})$, let $a$ be a loop based at $\overline{\star}$ representing $\alpha$, which does not meet singular points. There is a unique lift $\widetilde{a}$ of $a$ starting at $\star$ and ending at some point $\star'$. Then there exists a unique $h\in G$ so that $h(\star)=\star'$, and we have $\phi(\alpha)=h$. Similarly, pick a regular point $\ast$ in $\mathbb{R}^n$, which has image $\overline{\ast}$ in $\mathbb{R}^n/\rho(G)$. Then we have an exact sequence
\[1\rightarrow \pi_1(\mathbb{R}^n,\ast)\rightarrow \pi_1(\mathbb{R}^n/\rho(G),\overline{\ast})\rightarrow \rho(G)\rightarrow 1.\]
Now the map $\psi:\pi_1(\mathbb{R}^n/\rho(G),\overline{\ast})\rightarrow \rho(G)$ can be defined in the same way as for $\phi$, because $\mathrm{codim}\mathrm{Fix}(\rho(h))\geq 2$ for any nontrivial element $h$ in $G$.

In this paper, by an {\it embedding} $\overline{e}:\Sigma_g/G\rightarrow\mathbb{R}^n/\rho(G)$, we mean an embedding $\overline{e}$ between the underlying spaces $|\Sigma_g/G|$ and $|\mathbb{R}^n/\rho(G)|$ so that the points $\overline{x}\in\Sigma_g/G$ and $\overline{e}(\overline{x})\in\mathbb{R}^n/\rho(G)$ always have isomorphic local groups. Namely, for $x\in\Sigma_g$ and $y\in\mathbb{R}^n$ that map to $\overline{x}$ and $\overline{e}(\overline{x})$, respectively, we have $\mathrm{Stab}(x)\cong\mathrm{Stab}(y)$. For such an $\overline{e}$ with $\overline{e}(\overline{\star})=\overline{\ast}$, we have a homomorphism $\iota:\pi_1(\Sigma_g/G,\overline{\star})\rightarrow\pi_1(\mathbb{R}^n/\rho(G),\overline{\ast})$ as follows. For $\alpha\in\pi_1(\Sigma_g/G,\overline{\star})$, let $a$ be a loop representing $\alpha$ as before. Then, $\overline{e}\circ a$ is a loop representing some $\beta\in\pi_1(\mathbb{R}^n/\rho(G),\overline{\ast})$, and we have $\iota(\alpha)=\beta$. The above conditions on $\overline{e}$ ensure that the map $\iota$ is well-defined.

\begin{lemma}\label{lem:lift}
If there exists an embedding $\overline{e}:\Sigma_g/G\rightarrow\mathbb{R}^n/\rho(G)$ so that the map $\iota$ satisfies $\rho\circ\phi=\psi\circ\iota$, then $\overline{e}$ has a unique lift $e:\Sigma_g\rightarrow\mathbb{R}^n$ with $e(\star)=\ast$, which is a $G$-equivariant embedding with respect to $\rho$.
\end{lemma}

\begin{proof}
Let $\mathrm{Reg}(\Sigma_g)$ denote the subsurface of $\Sigma_g$ consisting of regular points. Then the map $\mathrm{Reg}(\Sigma_g)\rightarrow\mathrm{Reg}(\Sigma_g)/G$ is a covering map. So we have the exact sequence in the first row of the diagram, where the map $\phi':\pi_1(\mathrm{Reg}(\Sigma_g)/G,\overline{\star})\rightarrow G$ is given in the same way as for $\phi$. So $\phi'$ factors through $\phi$. Similarly we have $\mathrm{Reg}(\mathbb{R}^n)$ and the exact sequence in the second row, where the $\psi':\pi_1(\mathrm{Reg}(\mathbb{R}^n)/\rho(G),\overline{\ast})\rightarrow \rho(G)$ factors through $\psi$. Now $\overline{e}$ induces an embedding $\overline{e}':\mathrm{Reg}(\Sigma_g)/G\rightarrow\mathrm{Reg}(\mathbb{R}^n)/\rho(G)$, so we get the map $\iota'$ in the diagram. Then $\rho\circ\phi=\psi\circ\iota$ implies $\rho\circ\phi'=\psi'\circ\iota'$. So the image of $\pi_1(\mathrm{Reg}(\Sigma_g),\star)$ in $\pi_1(\mathrm{Reg}(\mathbb{R}^n)/\rho(G),\overline{\ast})$ is contained in $\mathrm{ker}(\psi')$, which equals the image of $\pi_1(\mathrm{Reg}(\mathbb{R}^n),\ast)$. Then, by covering space theory, we see that $\overline{e}'$ has a unique lift $e':\mathrm{Reg}(\Sigma_g)\rightarrow\mathrm{Reg}(\mathbb{R}^n)$ so that $e'(\star)=\ast$. \[\xymatrix{
  1 \ar[r] & \pi_1(\mathrm{Reg}(\Sigma_g),\star) \ar[r] & \pi_1(\mathrm{Reg}(\Sigma_g)/G,\overline{\star}) \ar[d]_{\iota'} \ar[r] & G \ar[d]_{\rho} \ar[r] & 1\\
  1 \ar[r] & \pi_1(\mathrm{Reg}(\mathbb{R}^n),\ast) \ar[r] & \pi_1(\mathrm{Reg}(\mathbb{R}^n)/\rho(G),\overline{\ast}) \ar[r] & \rho(G) \ar[r] & 1}\]

The map $e'$ is injective. Otherwise there are distinct points $x$ and $y$ in $\mathrm{Reg}(\Sigma_g)$ such that $e'(x)=e'(y)$. Since $\overline{e}'$ is an embedding, $x$ and $y$ have the same image in $\mathrm{Reg}(\Sigma_g)/G$. So there exists a nontrivial element $h$ in $G$ such that $h(x)=y$. Let $a$ be a path in $\mathrm{Reg}(\Sigma_g)$ from $x$ to $y$, and let $c$ be a path in $\mathrm{Reg}(\Sigma_g)$ from $\star$ to $x$. So $c'=h\circ c^{-1}$ is a path from $y$ to $\star'=h(\star)$, and we get a path $cac'$ from $\star$ to $\star'$. Its image in $\mathrm{Reg}(\Sigma_g)/G$ gives an element $\alpha\in\pi_1(\mathrm{Reg}(\Sigma_g)/G,\overline{\star})$. Then $\phi'(\alpha)=h$. On the other hand, because $e'\circ c$ and $e'\circ h\circ c$ have the same image in $\mathrm{Reg}(\mathbb{R}^n)/\rho(G)$ and $e'(x)=e'(y)$, we have $e'(\star)=e'(\star')=\ast$. So $e'\circ (cac')$ is a loop, and its image in $\mathrm{Reg}(\mathbb{R}^n)/\rho(G)$ gives the element $\iota'(\alpha)\in\pi_1(\mathrm{Reg}(\mathbb{R}^n)/\rho(G),\overline{\ast})$. This means that $\psi'\circ\iota'(\alpha)=\rho\circ\phi'(\alpha)=\rho(h)$ is trivial, which is a contradiction.

Now let $x$ be a singular point in $\Sigma_g$, let $\overline{x}$ be its image in $\Sigma_g/G$, and let $B$ be a sufficiently small $n$-ball centered at $\overline{e}(\overline{x})$. Since $\overline{e}$ is continuous, we can get a small disk $D$ centered at $\overline{x}$ so that $\overline{e}(D)\subseteq B$. Let $y_j$, where $1\leq j\leq s$, be the preimages of $\overline{e}(\overline{x})$ in $\mathbb{R}^n$. Then the preimage of $B$ in $\mathbb{R}^n$ consists of $s$ pairwise disjoint $n$-balls $B_1,\ldots,B_s$ centered at $y_1,\ldots,y_s$, respectively. The preimage of $D$ in $\Sigma_g$ is a union of disjoint disks. We let $D_x$ denote the disk containing $x$. Then $e'(D_x\setminus\{x\})\subseteq B_j$ for some $j$. Since $B$ can be sufficiently small, by sending $x$ to $y_j$ we can obtain an embedding of $D_x$ in $\mathbb{R}^n$. Moreover, the image of $D_x$ in $\mathbb{R}^n$ is exactly the preimage of $\overline{e}(D)$ in $B_j$, because $\overline{x}$ and $\overline{e}(\overline{x})$ have isomorphic local groups. Then we see that the injective map $e'$ extends to an embedding $e:\Sigma_g\rightarrow\mathbb{R}^n$.

By the construction, we see that the required lift $e$ is unique. Note that for any element $h\in G$ the map $\rho(h)^{-1}\circ e\circ h:\Sigma_g\rightarrow\mathbb{R}^n$ is also a lift of $\overline{e}$. To show that $e$ is $G$-equivariant, we only need to check that $\rho(h)^{-1}\circ e\circ h(\star)=\ast$ for every $h\in G$. Let $\star'=h(\star)$, let $\widetilde{a}$ be a path in $\mathrm{Reg}(\Sigma_g)$ from $\star$ to $\star'$, and let $a$ be the image of $\widetilde{a}$ in $\Sigma_g/G$. Then $a$ gives an element $\alpha$ in $\pi_1(\Sigma_g/G,\overline{\star})$, and $\phi(\alpha)=h$. Now $e\circ\widetilde{a}$ is a path in $\mathrm{Reg}(\mathbb{R}^n)$ from $\ast$ to $e(\star')$, and $\overline{e}\circ a$ is the image of $e\circ\widetilde{a}$ in $\mathbb{R}^n/\rho(G)$, which means that $\psi(\iota(\alpha))$ is the element in $\rho(G)$ sending $\ast$ to $e(\star')$. Then, we obtain the required equality, because $\psi(\iota(\alpha))=\rho(\phi(\alpha))=\rho(h)$.
\end{proof}

\begin{remark}\label{rem:lift}
In Lemma~\ref{lem:lift}, we can replace $\mathbb{R}^n$ with a general smooth manifold $\Sigma$, where $\rho$ should embed $G$ into the diffeomorphism group of $\Sigma$, and the singular set of $\Sigma$ should have codimension at least $2$. The proof is the same. In particular, one can choose $\Sigma=\Sigma_g$, which gives the classification theorem about $(\Sigma_g,G)$.
\end{remark}

\begin{proof}[Proof of Theorem~\ref{thm:exist}]
Assume that $|\Sigma_g/G|$ has genus $\bar{g}$, and there exist $l$ singular points with indices $n_1,\ldots,n_l$ in $\Sigma_g/G$, where $\bar{g}, l\geq 0$. Then according to Figure~\ref{fig:GroupP} we have a generating set $\{\alpha_j,\beta_j,\gamma_k\}$ for $\pi_1(\Sigma_g/G,\overline{\star})$, where $1\leq j\leq\bar{g}$, $1\leq k\leq l$.

\begin{figure}[h]
\includegraphics{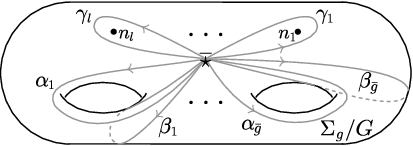}
\caption{The generators $\alpha_j$, $\beta_j$, $\gamma_k$ of $\pi_1(\Sigma_g/G,\overline{\star})$.}\label{fig:GroupP}
\end{figure}

Below we first construct an embedding $\overline{e}:\Sigma_g/G\rightarrow\mathbb{R}^n/\rho(G)$ with $\rho\circ\phi=\psi\circ\iota$. By Lemma~\ref{lem:lift}, this will give us a $G$-equivariant embedding $e$. Then we show that $e$ can be smooth if $\rho$ also satisfies the eigenvalue condition.

Let $a_j$, $b_j$, and $c_k$ denote those oriented circles in Figure~\ref{fig:GroupP} corresponding to $\alpha_j$, $\beta_j$, and $\gamma_k$, respectively. They intersect each other at $\overline{\star}$. Let $\Theta$ be a closed regular neighborhood of their union. It is a surface with $l+1$ boundary components. The complement of the interior of $\Theta$ in $|\Sigma_g/G|$ is a union of $l+1$ disks $D_0,D_1,\ldots,D_l$, where $D_k$ corresponds to $\gamma_k$ for $1\leq k\leq l$. Let $\epsilon>0$ be a small number, and let $\Xi$ be the complement in $\mathbb{R}^n/\rho(G)$ of the closed $\epsilon$-neighborhood of the singular set. It contains $\overline{\ast}$. We first embed $\Theta$ into $\Xi$, then we deal with those disks.

For each $1\leq k\leq l$, we pick a point $P_k$ in $\partial D_k$ and an arc $p_k$ from $\overline{\star}$ to $P_k$ as in Figure~\ref{fig:SurDk}, where an orientation of $\partial D_k$ is also given. So we can get a loop $\hat{c}_k$ going along $p_k\cup\partial D_k$ and representing $\gamma_k$. Note that the lift of $\hat{c}_k$ starting from $\star$ meets exactly one disk component of the preimage of $D_k$ in $\Sigma_g$. Let $D_x$ denote this disk, where $x$ is the singular point in the disk. Then it is not hard to see that $\phi(\gamma_k)$ lies in $\mathrm{Stab}(x)$, and its restriction to $D_x$ is conjugate to a $2\pi/n_k$-rotation.

\begin{figure}[h]
\includegraphics{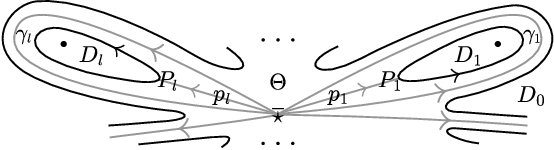}
\caption{The surface $\Theta$ and the disks $D_0,D_1,\ldots,D_l$.}\label{fig:SurDk}
\end{figure}

By the dimension inequality, if $\mathrm{Stab}(x)<G$, then $\mathrm{dim}\mathrm{Fix}(\rho(\mathrm{Stab}(x)))>0$, and we have
$\mathrm{dim}\mathrm{Fix}(\rho(\mathrm{Stab}(x)))>\mathrm{dim}\mathrm{Fix}(\rho(H))$ for any $\mathrm{Stab}(x)<H\leq G$. Then one can find a point $y$ in $\mathbb{R}^n$ so that $\mathrm{Stab}(y)=\rho(\mathrm{Stab}(x))$ and $\|y\|$ is arbitrarily large, where ``$\|\cdot\|$'' denotes the Euclidean norm. If $\mathrm{Stab}(x)=G$ and $\mathrm{Fix}(G)$ has at least two points, then $\mathrm{dim}\mathrm{Fix}(\rho(G))>0$ and we still have a point $y$ as above. If $\mathrm{Fix}(G)$ only contains $x$, then we just choose $y$ to be the origin in $\mathbb{R}^n$.

Let $B_y$ be the closed $2\epsilon$-neighborhood of $y$. Let $\overline{y}\in\mathbb{R}^n/\rho(G)$ be the image of $y$. One can further require that $\rho(h)(B_y)\cap B_y=\emptyset$ for any $h\in G\setminus\mathrm{Stab}(x)$. Then $B_y$ is $\rho(G)$-equivariant, and $B_y/\rho(\mathrm{Stab}(x))$ gives the closed $2\epsilon$-neighborhood of $\overline{y}$. Let $B$ denote this neighborhood. Then, because $\mathrm{codim}\mathrm{Fix}(\rho(h))\geq 2$ for any nontrivial element $h$ in $G$, the preimage of $\partial B\cap\Xi$ in $B_y$ is connected, and we can choose an arc in it from some point $R_k$ to $\rho(\phi(\gamma_k))(R_k)$. Since $n\geq 4$, we can further assume that the image of the arc in $\partial B\cap\Xi$ is a simple closed curve, denoted by $r_k$. Then let $Q_k\in r_k$ be the image of $R_k$, and let $q_k$ be an arc in $\Xi$ from $\overline{\ast}$ to $Q_k$ so that its lift starting from $\ast$ ends at $R_k$. By the construction, if we let $\delta_k$ be the element in $\pi_1(\mathbb{R}^n/\rho(G),\overline{\ast})$ given by a loop going along $q_k\cup r_k$, then $\psi(\delta_k)=\rho(\phi(\gamma_k))$.

Note that for different $1\leq k\leq l$, we will have different $\overline{y}$ and $B$. So we have $\overline{y}_k$ and $B_k$ for each $1\leq k\leq l$. By the above construction, one can require that all $B_k$ are disjoint, they meet each $q_k$ only at $Q_k$, and those $q_k$ meet only at $\overline{\ast}$. Also, one can require that near $\overline{\ast}$ all $q_k$ lie in a disk, as those $p_k$ near $\overline{\star}$. Then, it is not hard to see that there is an embedding $\overline{e}':\Theta\rightarrow\Xi$ sending $p_k\cup\partial D_k$ to $q_k\cup r_k$, sending $a_j$, $b_j$ to circles that correspond to $\rho(\phi(\alpha_j))$ and $\rho(\phi(\beta_j))$ respectively, and having image $\overline{e}'(\Theta)$ disjoint from the interiors of all those $B_k$.

Now we cap off $\overline{e}'(\partial D_k)=r_k$ by a cone in $B_k$ with vertex $\overline{y}_k$ for $1\leq k\leq l$. The codimension condition implies that $\overline{e}'(\partial D_0)$ is nullhomotopic in $\Xi$ and it bounds a disk disjoint from $\overline{e}'(\Theta\setminus\partial D_0)$ and the cones. So we have the required $\overline{e}$.

If $\rho$ also satisfies the eigenvalue condition, then, since in the above construction $\phi(\gamma_k)$ is a $2\pi/n_k$-rotation near $x$, there is a $\rho(\phi(\gamma_k))$-invariant $2$-dimensional plane $\Pi$ with $y\in\Pi$ so that the restriction of $\rho(\phi(\gamma_k))$ to $\Pi$ is also a $2\pi/n_k$-rotation. We note that the circle $\partial B_y\cap\Pi$ lies in the preimage of $\partial B\cap\Xi$. Hence one can choose the arc from $R_k$ to $\rho(\phi(\gamma_k))(R_k)$ to be an arc in $\partial B_y\cap\Pi$ of length $4\epsilon\pi/n_k$. So we see that the cones in those $B_k$ can have smooth preimages. Then, since $n>4$, the part of $\overline{e}(\Sigma_g/G)$ in $\Xi$ can also be smooth. So $e$ can be smooth.
\end{proof}


\subsection{Proofs of Proposition~\ref{thm:Cyclic} and Theorem~\ref{thm:GEBound}}\label{subsec:UB}

\begin{proof}[Proof of Proposition~\ref{thm:Cyclic}]
Given $n\geq 3$, we will construct a required $(\Sigma_g,G)$, where $g$ can be arbitrarily large. By the theory of orbifolds, we only need to construct a suitable epimorphism $\phi:\pi_1(\Sigma_g/G,\overline{\star})\rightarrow G$ that is injective on any finite subgroup of $\pi_1(\Sigma_g/G,\overline{\star})$, since $\pi_1(\Sigma_g,\star)$ is torsion free. We note that according to Figure~\ref{fig:GroupP}, $\pi_1(\Sigma_g/G,\overline{\star})$ has the following presentation, where $[\alpha_j,\beta_j]=\alpha_j\beta_j\alpha_j^{-1}\beta_j^{-1}$,
\[\langle\alpha_1,\beta_1,\ldots,\alpha_{\bar{g}},\beta_{\bar{g}},\gamma_1,\ldots,\gamma_l\mid \prod^{\bar{g}}_{j=1}[\alpha_j,\beta_j]\prod^{l}_{k=1}\gamma_k=1,\gamma_k^{n_k}=1,1\leq k\leq l\rangle.\]
Then since $G$ is cyclic, it suffices to focus on the values of $\phi(\gamma_k)$ and we see that $\bar{g}$ can be arbitrarily large. Let $G=\mathbb{Z}/m\mathbb{Z}$. There are three cases.

Case 1: $n=2t+1$ with $t>1$. Let $m=p$ be a prime number with $p\geq n$ and let $l=2t$. Then $n_k=p$ for $1\leq k\leq 2t$. We can define the $\phi$ so that $\phi(\gamma_{2k-1})=\overline{k}$ and $\phi(\gamma_{2k})=\overline{p-k}$ for $1\leq k\leq t$. This gives a required epimorphism. So one has some pair $(\Sigma_g,G)$. Below we show that $d_g(G)=2t+1$.

Suppose that there exists a $G$-equivariant embedding $e:\Sigma_g\rightarrow\mathbb{R}^s$, with respect to an embedding $\rho: G\rightarrow\mathrm{O}(s)$. Then, by Lemma~\ref{lem:nec}, $\mathrm{dim}\mathrm{Fix}(\rho(G))>0$, and each $\rho(\phi(\gamma_k))$ has the eigenvalues $\mathrm{e}^{\pm2\pi \mathrm{i}/p}$. For $1\leq k\leq t$, we can get $0<r_k<p$ so that $\overline{kr_k}=\overline{1}$. Then $\rho(\overline{1})$ has eigenvalues $\mathrm{e}^{\pm2r_1\pi \mathrm{i}/p},\ldots,\mathrm{e}^{\pm2r_t\pi \mathrm{i}/p}$, which are all different. Since $\mathrm{dim}\mathrm{Fix}(\rho(\overline{1}))\geq 1$, we must have $s\geq 2t+1$.

On the other hand, let $r_1,\ldots,r_t$ be as above. Then define $\rho: G\rightarrow\mathrm{O}(2t+1)$ so that $\rho(\overline{1})$ is the block diagonal matrix with the following blocks
\[
\begin{pmatrix}
  1 \\
\end{pmatrix} ,
\begin{pmatrix}
  \cos\frac{2r_1\pi \mathrm{i}}{p} & -\sin\frac{2r_1\pi \mathrm{i}}{p} \\
  \sin\frac{2r_1\pi \mathrm{i}}{p} & \cos\frac{2r_1\pi \mathrm{i}}{p} \\
\end{pmatrix} ,\ldots,
\begin{pmatrix}
  \cos\frac{2r_t\pi \mathrm{i}}{p} & -\sin\frac{2r_t\pi \mathrm{i}}{p} \\
  \sin\frac{2r_t\pi \mathrm{i}}{p} & \cos\frac{2r_t\pi \mathrm{i}}{p} \\
\end{pmatrix} .
\]
It is not hard to check that $\rho$ satisfies the dimension inequality and the eigenvalue condition. Since $t>1$, $\rho$ also satisfies the codimension condition. By Theorem~\ref{thm:exist}, we see that $d_g(G)\leq 2t+1$. So $d_g(G)=2t+1$.

Case 2: $n=2t+2$ with $t>1$. Let $m=3p$, where $p\geq 3t-1$ is a prime number so that $p\equiv 2\,(\mathrm{mod}\,3)$. Then, let $l=2t-1$, and let $n_1=3p$, $n_2=3$, $n_k=p$, where $3\leq k\leq 2t-1$. One can define the $\phi$ so that $\phi(\gamma_{1})=\overline{1}$, $\phi(\gamma_{2})=\overline{p}$, $\phi(\gamma_{3})=\overline{2p-1}$, $\phi(\gamma_{2k})=\overline{3k}$, and $\phi(\gamma_{2k+1})=\overline{3p-3k}$, where $2\leq k\leq t-1$. It is not hard to check that this gives a required epimorphism. Note that $2p-1\equiv 0\,(\mathrm{mod}\,3)$. So we have some pair $(\Sigma_g,G)$. Below we show that $d_g(G)=2t+2$.

Suppose that there are $e:\Sigma_g\rightarrow\mathbb{R}^s$ and $\rho: G\rightarrow\mathrm{O}(s)$ as above. By Lemma~\ref{lem:nec}, $\mathrm{dim}\mathrm{Fix}(\rho(\overline{3}))>\mathrm{dim}\mathrm{Fix}(\rho(\overline{1}))$, and for $3\leq k\leq 2t-1$, $\rho(\phi(\gamma_k))$ has the eigenvalues $\mathrm{e}^{\pm2\pi \mathrm{i}/p}$. Let $r_k$ be as above for $2\leq k\leq t-1$. Then, $\rho(\overline{3})$ has different eigenvalues $\mathrm{e}^{\pm2\pi \mathrm{i}/p}$, $\mathrm{e}^{\pm6\pi \mathrm{i}/p}$, $\mathrm{e}^{\pm2r_2\pi \mathrm{i}/p},\ldots,\mathrm{e}^{\pm2r_{t-1}\pi \mathrm{i}/p}$ where those first two come from $\gamma_1$ and $\gamma_3$. Since $\mathrm{dim}\mathrm{Fix}(\rho(\overline{3}))-\mathrm{dim}\mathrm{Fix}(\rho(\overline{1}))$ is even, we must have $s\geq 2t+2$.

On the other hand, let $r_2,\ldots,r_{t-1}$ be as above. Then, define $\rho: G\rightarrow\mathrm{O}(2t+2)$ so that $\rho(\overline{1})$ is the block diagonal matrix with the following blocks
\begin{align*}
&\begin{pmatrix}
  \cos\frac{2\pi \mathrm{i}}{3p} & -\sin\frac{2\pi \mathrm{i}}{3p} \\
  \sin\frac{2\pi \mathrm{i}}{3p} & \cos\frac{2\pi \mathrm{i}}{3p} \\
\end{pmatrix} ,
\begin{pmatrix}
  \cos\frac{2\pi \mathrm{i}}{3} & -\sin\frac{2\pi \mathrm{i}}{3} \\
  \sin\frac{2\pi \mathrm{i}}{3} & \cos\frac{2\pi \mathrm{i}}{3} \\
\end{pmatrix} ,
\begin{pmatrix}
  \cos\frac{2\pi \mathrm{i}}{p} & -\sin\frac{2\pi \mathrm{i}}{p} \\
  \sin\frac{2\pi \mathrm{i}}{p} & \cos\frac{2\pi \mathrm{i}}{p} \\
\end{pmatrix} ,\\
&\begin{pmatrix}
  \cos\frac{2r_2\pi \mathrm{i}}{3p} & -\sin\frac{2r_2\pi \mathrm{i}}{3p} \\
  \sin\frac{2r_2\pi \mathrm{i}}{3p} & \cos\frac{2r_2\pi \mathrm{i}}{3p} \\
\end{pmatrix} ,\ldots,
\begin{pmatrix}
  \cos\frac{2r_{t-1}\pi \mathrm{i}}{3p} & -\sin\frac{2r_{t-1}\pi \mathrm{i}}{3p} \\
  \sin\frac{2r_{t-1}\pi \mathrm{i}}{3p} & \cos\frac{2r_{t-1}\pi \mathrm{i}}{3p} \\
\end{pmatrix} .
\end{align*}
Note that $G$ is cyclic. Hence we have $\mathrm{Stab}(x)=\mathrm{Stab}(x')$ for any singular points $x$ and $x'$ in $\Sigma_g$ that lie in the same orbit, and for any element $h$ in $\mathrm{Stab}(x)$, $dh_x$ and $dh_{x'}$ have the same eigenvalues. Then one can check that $\rho$ satisfies the dimension inequality, eigenvalue condition, and codimension condition. And by Theorem~\ref{thm:exist}, we see that $d_g(G)\leq 2t+2$. So $d_g(G)=2t+2$.

Case 3: $n\leq 4$. It is obvious when $n=3$. For $n=4$ we give an explicit example as follows. Let $\Sigma_{1,4}\subset\mathbb{R}^3$ denote the torus with four punctures shown in Figure~\ref{fig:N4}. It is obtained from two annuli in $\mathbb{R}^2\subset\mathbb{R}^3$ by adding four half-twisted bands. Now write $\mathbb{R}^4=\mathbb{R}^3\times\mathbb{R}=\mathbb{R}^2\times\mathbb{R}\times\mathbb{R}$. Then the block diagonal matrix with blocks
\[
\begin{pmatrix}
  0 & -1 \\
  1 & 0 \\
\end{pmatrix} ,
\begin{pmatrix}
  -1 \\
\end{pmatrix} ,
\begin{pmatrix}
  1 \\
\end{pmatrix}
\]
keeps $\Sigma_{1,4}$ invariant. The corresponding orbifold is a thrice-punctured sphere. See the right picture in Figure~\ref{fig:N4}. Since one can cap off $\partial\Sigma_{1,4}$ equivariantly by smooth disks in $\mathbb{R}^4$, we have a pair $(\Sigma_1,G)$ with $G=\mathbb{Z}/4\mathbb{Z}$ and $d_1(G)\leq 4$.

\begin{figure}[h]
\includegraphics{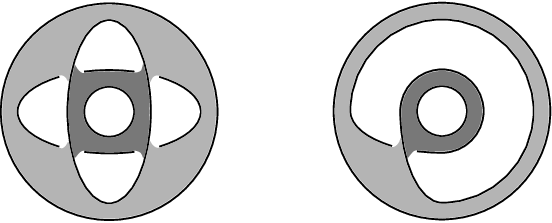}
\caption{The surface $\Sigma_{1,4}$ and the orbifold $\Sigma_{1,4}/G$.}\label{fig:N4}
\end{figure}

For $\Sigma_1/G$, we have $\bar{g}=0$ and $l=3$. We can assume that $(n_1,n_2,n_3)=(4,4,2)$, $\phi(\gamma_1)=\phi(\gamma_2)=\overline{1}$, and $\phi(\gamma_3)=\overline{2}$. Suppose we have $e:\Sigma_1\rightarrow\mathbb{R}^s$ and $\rho: G\rightarrow\mathrm{O}(s)$ as before. Then by Lemma~\ref{lem:nec}, $\mathrm{dim}\mathrm{Fix}(\rho(\overline{2}))>\mathrm{dim}\mathrm{Fix}(\rho(\overline{1}))>0$ and $\rho(\overline{2})$ has $-1$ as a multiple eigenvalue. So $s\geq 4$ and we get $d_1(G)=4$. Clearly, by adding small handles equivariantly, we can get $(\Sigma_g,G)$ with $d_g(G)=4$ for large $g$.

Finally, if $n\neq 3$ is a prime number, then we can choose $m=n$ in Case~1.
\end{proof}

\begin{proof}[Proof of Theorem~\ref{thm:GEBound}]
Let $\rho: G\rightarrow\mathrm{O}(|G|)$ be the regular representation. We need to check that $\rho$ actually satisfies the three conditions in Theorem~\ref{thm:exist}.

Note that $\rho$ is a permutation representation. Each $h$ in $G$ permutes elements of $G$, and the number of orbits is equal to $\mathrm{dim}\mathrm{Fix}(\rho(h))$. If $h$ is a nontrivial element, then no elements of $G$ can be fixed. So each orbit has at least two elements. Since $|G|>4$, we have $\mathrm{dim}\mathrm{Fix}(\rho(h))<|G|-2$. So the codimension condition holds.

For $x\in\Sigma_g$ and $H\leq G$ such that $\mathrm{Stab}(x)<H$, we can pick $h\in H\setminus\mathrm{Stab}(x)$. It is clear that $\mathrm{Fix}(\rho(\mathrm{Stab}(x)))\supseteq\mathrm{Fix}(\rho(H))$. Let $\mathrm{Stab}(x)=\{h_1,\ldots,h_s\}$. Then since elements of $G$ can be viewed as vectors in $\mathbb{R}^{|G|}$, we can define $\eta=\sum_{r=1}^{s}h_r$. So we have $\rho(h)(\eta)=\sum_{r=1}^{s}hh_r\neq\eta$, and $\eta\in\mathrm{Fix}(\rho(\mathrm{Stab}(x)))\setminus\mathrm{Fix}(\rho(H))$. Note that we always have $\mathrm{dim}\mathrm{Fix}(\rho(G))=1$. So the dimension inequality holds.

Now assume that $x\in\Sigma_g$ is a singular point, and $h$ is a generator of $\mathrm{Stab}(x)$. If $h$ has order $s$, then it is easy to see that $\rho(h)$ contains all those $s$-th roots of unity as eigenvalues. Also, note that if $s=2$, then $\rho(h)$ has $-1$ as a multiple eigenvalue, because $|G|\geq 4$. So the eigenvalue condition also holds.
\end{proof}

\begin{remark}\label{rem:smallG}
The inequality in Theorem~\ref{thm:GEBound} does not hold if $|G|\leq 4$. Actually, we have $d_g(G)=3$ if $|G|=2$, by the classification of involutions on $\Sigma_g$. Moreover, we have $(\Sigma_g,G)$ so that $G=\mathbb{Z}/m\mathbb{Z}$ and $d_g(G)=5$, for each $m\geq 3$, as follows.

Assume that for $\Sigma_g/G$ we have $\bar{g}=0$, $l=m$, and $n_k=m$ for $1\leq k\leq l$. We let $\phi$ satisfy $\phi(\gamma_{k})=\overline{1}$ for $1\leq k\leq l$. As before, one can obtain $(\Sigma_g,G)$ corresponding to $\phi$, and if there are $e:\Sigma_g\rightarrow\mathbb{R}^s$ and $\rho: G\rightarrow\mathrm{O}(s)$, then by Lemma~\ref{lem:nec}, $\rho(\overline{1})$ has the eigenvalues $1$ and $\mathrm{e}^{\pm2\pi \mathrm{i}/m}$. Now if $s\leq 4$, then those singular points in $\mathbb{R}^s$ give a subspace of codimension $2$, and we always have $\pi_1(\mathrm{Reg}(\mathbb{R}^s)/\rho(G),\overline{\ast})\cong\mathbb{Z}$. Then, let $D_k$ be as in the proof of Theorem~\ref{thm:exist}, where $1\leq k\leq l$. Since for each singular point $x\in\Sigma_g$, $de_x(T_x\Sigma_g)$ must coincide with the eigenspace for $\mathrm{e}^{\pm2\pi \mathrm{i}/m}$, the image of the oriented $\partial D_k$ in $\mathrm{Reg}(\mathbb{R}^s)/\rho(G)$ gives a generator of $\pi_1(\mathrm{Reg}(\mathbb{R}^s)/\rho(G))$. The generator is independent of $k$, since its image in $\rho(G)$ is $\rho(\overline{1})$. This contradicts the fact that the union of the images of $\partial D_k$ in $\mathrm{Reg}(\mathbb{R}^s)/\rho(G)$ is null-homologous. On the other hand, let $\rho(\overline{1})$ be the block diagonal matrix with the following blocks
\[
\begin{pmatrix}
  1 \\
\end{pmatrix} ,
\begin{pmatrix}
  \cos\frac{2\pi \mathrm{i}}{m} & -\sin\frac{2\pi \mathrm{i}}{m} \\
  \sin\frac{2\pi \mathrm{i}}{m} & \cos\frac{2\pi \mathrm{i}}{m} \\
\end{pmatrix} ,
\begin{pmatrix}
  \cos\frac{2\pi \mathrm{i}}{m} & -\sin\frac{2\pi \mathrm{i}}{m} \\
  \sin\frac{2\pi \mathrm{i}}{m} & \cos\frac{2\pi \mathrm{i}}{m} \\
\end{pmatrix} .
\]
Then we have $\rho: G\rightarrow\mathrm{O}(5)$, and by Theorem~\ref{thm:exist}, we see that $d_g(G)=5$.

It is also easy to see that $d_g(G)\leq 5$ if $|G|=3$, and the inequality is sharp.
\end{remark}


\section{Equivariant triangulations and pseudo-triangulations}\label{sec:ETandPT}
Given $(\Sigma_g,G)$, in Section~\ref{subsec:ETtoEE} and Section~\ref{subsec:PTof2O} we will first show how to obtain a required $\rho: G\rightarrow\mathrm{O}(n)$ from a $G$-equivariant triangulation of $\Sigma_g$ or from a pseudo-triangulation of $\Sigma_g/G$. Then, in Section~\ref{subsec:LAandHTG} we will investigate those large actions and the related hyperbolic triangle groups, and we will prove Theorem~\ref{thm:SLBound}.

\subsection{From equivariant triangulations to equivariant embeddings}\label{subsec:ETtoEE}
Given a simplicial complex $K$, we will use $K^i$ and $|K|$ to denote the $i$-th skeletons and the underlying space of $K$, respectively, where $i\geq 0$. For a vertex $v$ in $K^0$, we will use $\mathrm{St}(v)$ and $\mathrm{Lk}(v)$ to denote the star and link of $v$ in $K$, respectively.

\begin{definition}
A {\it triangulation} of $\Sigma_g$ is a pair $(K,\tau)$, such that $K$ is a simplicial complex and $\tau:|K|\rightarrow\Sigma_g$ is a homeomorphism. It is {\it $G$-equivariant} if $\tau^{-1}h\tau$ gives a simplicial map for each element $h\in G$.
\end{definition}

\begin{theorem}\label{thm:tri}
Given $(\Sigma_g,G)$, if there exists a $G$-equivariant triangulation $(K,\tau)$ of $\Sigma_g$ so that $K^0$ can be written as a disjoint union of two $G$-invariant sets $U$ and $V$, where $\mathrm{St}(v)\cap\mathrm{St}(w)=\mathrm{Lk}(v)\cap\mathrm{Lk}(w)$ and $\mathrm{Lk}(v)^0\neq\mathrm{Lk}(w)^0$ hold for any $v\neq w$ in $U$, then there exists a smooth $G$-equivariant embedding $e:\Sigma_g\rightarrow\mathbb{R}^{|V|}$.
\end{theorem}

\begin{proof}
We identify $\Sigma_g$ with $|K|$ via $\tau$. By the conditions, it is not hard to see that $|V|\geq 3$. Assume that $V=\{v_1,\ldots,v_n\}$, and let $\{e_1,\ldots,e_n\}$ be the standard basis for $\mathbb{R}^{|V|}$. We can define a map $\hat{e}:\Sigma_g\rightarrow\mathbb{R}^{|V|}$ as follows. First we let $\hat{e}(v_j)=e_j$ for $1\leq j\leq n$. Then let $\hat{e}(u)=\sum\hat{e}(v)$ for each $u\in U$, where the sum is over all those $v\in\mathrm{Lk}(u)^0$. Note that $\mathrm{Lk}(u)^0\subseteq V$ by the conditions. So we have $K^0\rightarrow\mathbb{R}^{|V|}$. The required $\hat{e}$ is then the linear extension of this map.

Clearly $\hat{e}(u)\neq\hat{e}(v)$ for $u\in U, v\in V$. Since $\mathrm{Lk}(v)^0\neq\mathrm{Lk}(w)^0$ for $v\neq w$ in $U$, we also have $\hat{e}(v)\neq\hat{e}(w)$. So $\hat{e}:K^0\rightarrow\mathbb{R}^{|V|}$ is injective. Then $\hat{e}:\sigma\rightarrow\mathbb{R}^{|V|}$ is a linear embedding for any $2$-simplex $\sigma$ in $K$. Suppose that $\sigma\neq\sigma'$, which have vertex sets $\{u,v,w\}\neq\{u',v',w'\}$, but there exists $y\in\hat{e}(\sigma)\cap\hat{e}(\sigma')$. Then $y$ is given by
\[r\hat{e}(u)+s\hat{e}(v)+t\hat{e}(w)=r'\hat{e}(u')+s'\hat{e}(v')+t'\hat{e}(w')\]
for some unique $r,s,t,r',s',t'\geq 0$ satisfying $r+s+t=r'+s'+t'=1$. Note that any $2$-simplex has at most one vertex lying in $U$. If both $\sigma$ and $\sigma'$ contain no such vertices, then, by the vector equation, $\hat{e}(\sigma)$ and $\hat{e}(\sigma')$ meet in a common edge or a common vertex. Otherwise, assume that $u\in U$, then $v,w\in\mathrm{Lk}(u)^0\subseteq V$. If $r\neq 0$, then by the definition of $\hat{e}(u)$, all the vertices in $\mathrm{Lk}(u)^0$ will appear in the left side of the vector equation, and the sum of the coefficients is $r|\mathrm{Lk}(u)^0|+s+t>1$. So, $\{u',v',w'\}\cap U\neq\emptyset$. Assume that $u'\in U$. Then $r'\neq 0$, and by a similar reason we have $\mathrm{Lk}(u')^0=\mathrm{Lk}(u)^0$. So $u'=u$, and $r'=r$. As before, $\hat{e}(\sigma)$ and $\hat{e}(\sigma')$ meet in a common edge or a common vertex. The case when $u\in U$ and $r=0$ will meet the same conclusion. Hence $\hat{e}:\Sigma_g\rightarrow\mathbb{R}^{|V|}$ is an embedding.

Now, we define a representation $\rho: G\rightarrow\mathrm{O}(n)$. Note that $\hat{e}(v_j)=e_j$, $1\leq j\leq n$, give a basis for $\mathbb{R}^{|V|}$ and $V$ is $G$-invariant. So the relations $\rho(h)(\hat{e}(v_j))=\hat{e}(h(v_j))$, $1\leq j\leq n$, provide an element $\rho(h)\in\mathrm{O}(n)$ for each $h\in G$, and hence define the $\rho$. Since $u\in U$ is determined by $\mathrm{Lk}(u)^0\subseteq V$, we see that $\rho$ is an embedding, and $\hat{e}$ is a $G$-equivariant embedding with respect to $\rho$.

Below we also need to show that $\hat{e}$ can be modified to give the required smooth embedding $e$. We can assume that $|V|>4$, otherwise $\Sigma_g$ is a $2$-sphere, and in this case the theorem always holds. Also, note that for each nontrivial element $h$ in $G$, $\rho(h)$ permutes $\{e_1,\ldots,e_n\}$, and the number of orbits is equal to $\mathrm{dim}\mathrm{Fix}(\rho(h))$.

The embedded surface $\hat{e}(\Sigma_g)$ consists of some triangles in $\mathbb{R}^{|V|}$. By Lemma~\ref{lem:lift}, we can assume that the points in $\Sigma_g$ where $\hat{e}$ is not smooth correspond to vertices and edges of these triangles, and if $\hat{e}$ can be modified so that it is smooth near the singular points, then it can be modified to give the required $e$. So we only need to consider the singular points in $\hat{e}(|K^1|)$. Let $y$ be such a point. Then since $\mathrm{Stab}(y)$ is cyclic, we can pick a generator $\rho(h)$ of it, where $h\in G$ is nontrivial.

According to whether $y$ lies in $\hat{e}(U)$ or $\hat{e}(V)$, there are three cases.

Case 1: $y\notin\hat{e}(K^0)$. Then, $\rho(h)$ has order $2$, and $y$ is the midpoint of some edge, whose endpoints must lie in $\hat{e}(V)$. Assume that $\hat{e}(v_1)$ and $\hat{e}(v_2)$ are the endpoints. Let $\hat{e}(u_1)$ and $\hat{e}(u_2)$ be the other two vertices of the triangles containing $y$. If they lie in $\hat{e}(U)$, then one can find $w_1$ and $w_2$ in $V$, so that $w_1\in\mathrm{Lk}(u_1)^0\setminus\mathrm{Lk}(u_2)^0$ and $w_2=h(w_1)$. So we can assume that $\{v_3,v_4\}$ is either $\{u_1,u_2\}$ or $\{w_1,w_2\}$, and we see that $\rho(h)$ interchanges each of $\{\hat{e}(v_1),\hat{e}(v_2)\}$ and $\{\hat{e}(v_3),\hat{e}(v_4)\}$.

Let $\Pi$ be the $2$-dimensional plane containing $y$ and spanned by $e_1-e_2$, $e_3-e_4$. Let $D_y$ be the union of the triangles containing $y$. Then the projection $D_y\rightarrow\Pi$ is an embedding. So the part of $D_y$ near $y$ can be modified in an equivariant way so that it lies in $\Pi$, and the modified $\hat{e}$ can be smooth near $\hat{e}^{-1}(y)$.

Case 2: $y\in\hat{e}(U)$. Assume that $y=\hat{e}(u)$, and $D_y$ is the union of those triangles containing $y$, where $\mathrm{Lk}(u)^0=\{v_1,\ldots,v_k\}\subseteq V$, and along $\partial D_y$ we have $e_1,\ldots,e_k$. Then $k\geq 3$. It is easy to see that the method in Case~1 also works for $k\leq 4$. The plane $\Pi$ is spanned by $e_1-e_2$, $e_2-e_3$ if $k=3$, and by $e_1-e_3$, $e_2-e_4$ if $k=4$. It actually also works for $k\geq 5$, but we prefer another approach as follows.

Let $B_y$ be the closed $2\epsilon$-neighborhood of $y$. Now $\mathrm{Stab}(y)$ acts on $B_y$. Since $h$ is a rotation near $u$, a nontrivial element in $\mathrm{Stab}(y)$ can fix no points in $\{e_1,\ldots,e_k\}$. So the number of the corresponding orbits is at most $k/2$. Then, for $k\geq 5$, we see that the singular set in $B_y$ has dimension at most $n-3$. On the other hand, if we let $s$ be the order of $h$, then $\rho(h)$ has all the $s$-th roots of unity as eigenvalues and has $-1$ as a multiple eigenvalue if $s=2$. Assume that $h$ is a $2\pi/s$-rotation near $u$. Then there is a $\rho(h)$-invariant plane $\Pi$ containing $y$, so that the restriction of $\rho(h)$ to $\Pi$ is also a $2\pi/s$-rotation. Let $B_y'$ be the open $\epsilon$-neighborhood of $y$. Then, as in the proof of Theorem~\ref{thm:exist}, we can replace the part of $D_y$ in $B_y'$ by a disk in $\Pi$, and replace the part of $D_y$ in $B_y\setminus B_y'$ by an embedded annulus, in an equivariant way. So we see that the modified $\hat{e}$ can be smooth near $\hat{e}^{-1}(y)$.

Case 3: $y\in\hat{e}(V)$. Let $y=\hat{e}(v)$ and let $D_y$ be as above. If $\mathrm{Lk}(v)^0\subseteq V$, then we meet the same situation as in Case~2. Assume that $\mathrm{Lk}(v)^0\cap V=\{v_1,\ldots,v_k\}$ and along $\partial D_y$ there are some points in $\hat{e}(U)$ lying between $e_1,\ldots,e_k$, where $k\geq 2$. If $k\geq 5$, then the method in Case~2 also works. If $k=2$, then a similar argument as in Case~1 can be applied. So we only need to consider the remaining cases.

If $k=3$, then $\rho(h)$ has order $3$, and one can assume that $|\mathrm{St}(v)|$ is given by the left picture of Figure~\ref{fig:Star}, where $u_1,u_2,u_3\in U$. If $v_1\notin\mathrm{Lk}(u_1)^0$, then for $j=2,3$, we have $v_j\notin\mathrm{Lk}(u_j)^0$. In this case, the method in Case~1 also works, where the plane $\Pi$ is spanned by $e_1-e_2$, $e_2-e_3$. Otherwise $v_1,v_2,v_3\in\mathrm{Lk}(u_j)^0$ for each $j$. So one can find $w\in V\setminus\{v_1,v_2,v_3\}$, which is not fixed by $h$. This means that the number of $\rho(h)$-orbits in $\{e_1,\ldots,e_n\}$ is at most $n-4$. So the method in Case~2 works.

\begin{figure}[h]
\includegraphics{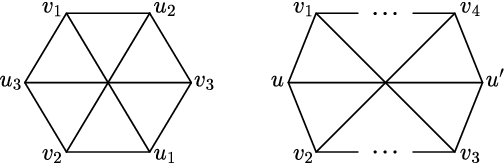}
\caption{The disks $|\mathrm{St}(v)|$ when $k=3$ and $k=4$.}\label{fig:Star}
\end{figure}

If $k=4$, then $\rho(h)$ has order $2$ or $4$. Assume that along $\partial D_y$ one has $\hat{e}(u)$ lying between $e_1$ and $e_2$, where $u\in U$. Then there is $u'\in U$ with $\hat{e}(u')$ lying between $e_3$ and $e_4$. See the right picture of Figure~\ref{fig:Star}. If $v_1,v_2,v_3,v_4\in\mathrm{Lk}(u)^0$, then they lie in $\mathrm{Lk}(u')^0$, and as above, one can find $w\in V\setminus\{v_1,v_2,v_3,v_4\}$ that is not fixed by the element of order $2$ in $\mathrm{Stab}(y)$. So, for this order $2$ element, the number of orbits is at most $n-3$. Then, the method in Case~2 also works. Otherwise, we can assume that no point in $\mathrm{Lk}(v)^0\cap U$ is adjacent to all $v_1,v_2,v_3,v_4$. Now let $\Pi$ be the plane containing $y$ and spanned by $e_1-e_3$, $e_2-e_4$, and consider the map $D_y\rightarrow\Pi$ as in Case~1. Even if it is not an embedding, it will become an embedding if we isotope $D_y$ suitably in an equivariant way. So the part of $D_y$ near $y$ can be moved into $\Pi$, and the modified $\hat{e}$ can be smooth near $\hat{e}^{-1}(y)$, as before.
\end{proof}

\begin{remark}\label{rem:tri}
As a special case, the conditions hold automatically if $U=\emptyset$. So one can get a smooth $G$-equivariant embedding $e:\Sigma_g\rightarrow\mathbb{R}^n$ where $n$ is the number of vertices of the triangulation. On the other hand, if the disks $|\mathrm{St}(u)|$, where $u\in U$, cover $\Sigma_g$, then we say that they give an {\it umbrella tessellation} of the surface $\Sigma_g$, or simply a {\it $U$-tessellation}. In this case, we can get relatively small dimensions. Also, we note that if one does not require $e$ to be smooth, then the result actually holds for general pairs $(M,G)$, where the proof is similar.
\end{remark}


\subsection{Pseudo-triangulations of the $2$-dimensional orbifolds}\label{subsec:PTof2O}

\begin{definition}
Let $\Sigma_g/G$ be the $2$-dimensional orbifolds given by $(\Sigma_g,G)$.

A {\it pseudo-triangle} in $\Sigma_g/G$ is a disk $\Delta\subset|\Sigma_g/G|$ homeomorphic to a triangle in $\mathbb{R}^2$, where corresponding to the edges and vertices of the triangle, it also has three edges and three vertices, and it can meet singular points only in the vertices.

A {\it pseudo-triangulation} of $\Sigma_g/G$ is a collection of finitely many pseudo-triangles in $\Sigma_g/G$ together with all their edges and vertices, denoted by $\mathcal{T}$, so that

(1) the pseudo-triangles in $\mathcal{T}$ cover $|\Sigma_g/G|$;

(2) the intersection of any two pseudo-triangles in $\mathcal{T}$ is either empty or a union of their common edges and common vertices.

A {\it pair of multiple edges} in $\mathcal{T}$ is a pair of edges with the same vertices.
\end{definition}

In Figure~\ref{fig:EofPT}, we show some examples when two pseudo-triangles have nonempty intersections. In each picture, there exist multiple edges.

\begin{figure}[h]
\includegraphics{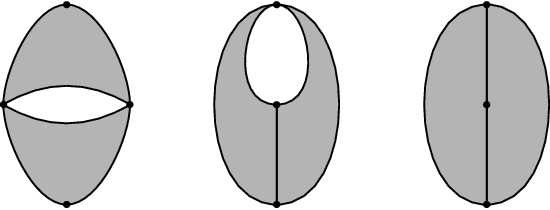}
\caption{Examples of pseudo-triangles with intersections.}\label{fig:EofPT}
\end{figure}

Note that manifolds are orbifolds. So we also have pseudo-triangulations of $\Sigma_g$. Then it is clear that a triangulation of $\Sigma_g$ gives a pseudo-triangulation of $\Sigma_g$. The following lemma shows when a pseudo-triangulation can give a triangulation.

\begin{lemma}\label{lem:PTtoT}
If a pseudo-triangulation of $\Sigma_g$ has no multiple edges, then either it gives a triangulation $(K,\tau)$ of $\Sigma_g$ where the $2$-simplexes correspond to the pseudo-triangles via $\tau$, or it is a pseudo-triangulation of $S^2$ with two pseudo-triangles.
\end{lemma}

\begin{proof}
Let $\mathcal{T}$ be the pseudo-triangulation. Then because $\mathcal{T}$ has no multiple edges, each edge in $\mathcal{T}$ is determined by its two vertices. If every pseudo-triangle in $\mathcal{T}$ can also be determined by its vertices, then one can construct a simplicial complex $K$, whose simplexes correspond to the elements in $\mathcal{T}$. And there is a homeomorphism $\tau:|K|\rightarrow\Sigma_g$ sending the $2$-simplexes to pseudo-triangles. Otherwise there are two pseudo-triangles $\Delta_1$ and $\Delta_2$ in $\mathcal{T}$ with the same vertices. Then they also have the same edges. So $\Sigma_g=\Delta_1\cup\Delta_2$ gives a pseudo-triangulation of $S^2$.
\end{proof}

Now let $\mathcal{T}$ be a pseudo-triangulation of $\Sigma_g/G$. Let $\Delta\in\mathcal{T}$ be a pseudo-triangle. Then, the preimage of $\Delta$ in $\Sigma_g$ is a union of $|G|$ pseudo-triangles, which may meet each other only in the vertices. So it is not hard to see that such pseudo-triangles, together with their edges and vertices, give a pseudo-triangulation of $\Sigma_g$. Since by the theory of orbifolds, the pair $(\Sigma_g,G)$ is determined by $\phi:\pi_1(\Sigma_g/G,\overline{\star})\rightarrow G$, we will use $\mathcal{T}_{\phi}$ to denote this pseudo-triangulation of $\Sigma_g$.

\begin{lemma}\label{lem:EofET}
If $G$ is nontrivial, then $\mathcal{T}_{\phi}$ gives a $G$-equivariant triangulation of $\Sigma_g$ if and only if it has no multiple edges.
\end{lemma}

\begin{proof}
The ``only if'' part is clear. Assume that $\mathcal{T}_{\phi}$ has no multiple edges. Since $G$ is nontrivial, $\mathcal{T}_{\phi}$ has more than two pseudo-triangles. So by Lemma~\ref{lem:PTtoT}, $\mathcal{T}_{\phi}$ gives a triangulation $(K,\tau)$ of $\Sigma_g$. We need to show that $(K,\tau)$ can be $G$-equivariant.

Since $\mathcal{T}_{\phi}$ comes from $\mathcal{T}$ and the map $\Sigma_g\rightarrow\Sigma_g/G$, we see that $G$ acts on the set of vertices in $\mathcal{T}_{\phi}$. So $G$ acts on $|K|$ via simplicial maps. This gives a pair $(|K|,G)$. Then we need to find $\tau:|K|\rightarrow\Sigma_g$ so that it gives a conjugation between $(|K|,G)$ and $(\Sigma_g,G)$. By Lemma~\ref{lem:lift} and Remark~\ref{rem:lift}, one only needs to find some suitable embedding $\overline{\tau}:|K|/G\rightarrow\Sigma_g/G$. Such $\overline{\tau}$ can be constructed by mapping the images of the simplexes in $|K|/G$ to the corresponding elements in $\mathcal{T}$.
\end{proof}

In what follows, we will assume that $\mathcal{T}$ has no multiple edges, which is the case we need later. Now, if $\mathcal{T}_{\phi}$ has a pair of multiple edges, then they map to the same edge in $\mathcal{T}$. For an edge $\eta$ in $\mathcal{T}$, its preimage in $\Sigma_g$ is a graph, where the connected components are isomorphic to each other via the elements in $G$. So we can use $\Lambda_{\eta}$ to denote such a component. And we see that $\mathcal{T}_{\phi}$ has multiple edges if and only if $\Lambda_{\eta}$ has multiple edges for some $\eta$ in $\mathcal{T}$.

\begin{lemma}\label{lem:mulE}
Given an edge $\eta\in\mathcal{T}$ and $\Lambda_{\eta}$, let $H_{\eta}=\{h\in G\mid h(\Lambda_{\eta})=\Lambda_{\eta}\}$. Then $\Lambda_{\eta}$ has multiple edges if and only if there exists a nontrivial element $h$ lying in the center of $H_{\eta}$ so that $h$ fixes all the vertices of $\Lambda_{\eta}$.
\end{lemma}

\begin{proof}
Let $\widetilde{\eta}$ and $\widetilde{\eta}'$ be a pair of edges in $\Lambda_{\eta}$. Then, there is a nontrivial element $h$ in $H_{\eta}$ so that $h(\widetilde{\eta})=\widetilde{\eta}'$. If they have the same vertices $x$ and $x'$, then $h$ fixes such vertices, and so $h\in\mathrm{Stab}(x)\cap\mathrm{Stab}(x')$. Hence, by fixing $\widetilde{\eta}$, there exists a bijective correspondence between the edges with the same vertices as $\widetilde{\eta}$ and the elements in $\mathrm{Stab}(\partial\widetilde{\eta})=\mathrm{Stab}(x)\cap\mathrm{Stab}(x')$. On the other hand, if $\widetilde{\eta}\cap\widetilde{\eta}'=x$, then $h\in\mathrm{Stab}(x)$ and we have $\mathrm{Stab}(\partial\widetilde{\eta})=h^{-1}\mathrm{Stab}(\partial\widetilde{\eta}')h= \mathrm{Stab}(\partial\widetilde{\eta}')$, since $\mathrm{Stab}(x)$ is cyclic. Then all the possible groups $\mathrm{Stab}(\partial\widetilde{\eta})$ are the same, because $\Lambda_{\eta}$ is connected. Note that $H_{\eta}$ can be generated by those groups $\mathrm{Stab}(x)$ corresponding to the vertices in $\Lambda_{\eta}$. So $\mathrm{Stab}(\partial\widetilde{\eta})$ lies in the center of $H_{\eta}$. And we have the result.
\end{proof}

The condition in the above lemma can also be described from the point of view of orbifolds. Without loss of generality, we can assume that the basepoint $\overline{\star}$ lies in the interior of $\eta$, and $\Lambda_{\eta}$ contains the basepoint $\star$. Then, $\overline{\star}$ divides $\eta$ into two arcs $\eta_1$ and $\eta_2$. For $i=1,2$, let $N_i$ be a small regular neighborhood of $\eta_i$ and let $H_i$ be the image of $\pi_1(N_i,\overline{\star})$ in $G$ under $\phi$. Let $H_0\subseteq G$ be the subgroup that consists of the elements fixing all the vertices of $\Lambda_{\eta}$. Then $H_0,H_1,H_2\subseteq H_{\eta}$.

\begin{lemma}\label{lem:theHi}
Given an edge $\eta\in\mathcal{T}$, let $H_{\eta}$, $H_1$, $H_2$, and $H_0$ be as above. Then $H_{\eta}$ is generated by $H_1$ and $H_2$, and $H_0=H_1\cap H_2$.
\end{lemma}

\begin{proof}
We can assume that $N=N_1\cup N_2$ gives a small regular neighborhood of $\eta$. Then the subgroup of $G$ generated by $H_1$ and $H_2$ is equal to $\phi(\pi_1(N,\overline{\star}))$.

Let $N'$ be the regular neighborhood of $\Lambda_{\eta}$ in $\Sigma_g$ corresponding to $N$. Then, for each $h\in H_{\eta}$, we can choose a path $\widetilde{a}$ in $N'$ from $\star$ to $h(\star)$, which gives a loop $a$ in $N$. So $H_{\eta}\subseteq\phi(\pi_1(N,\overline{\star}))$, by the definition of $\phi$, and we have $H_{\eta}=\phi(\pi_1(N,\overline{\star}))$.

Let $\widetilde{\eta}$ be the edge in $\Lambda_{\eta}$ containing $\star$. Then by the definition of $\phi$, those groups $H_1$ and $H_2$ are the stabilizers of those two endpoints of $\widetilde{\eta}$, respectively. And so we have $H_1\cap H_2=\mathrm{Stab}(\partial\widetilde{\eta})$, which is equal to $H_0$, by the proof of Lemma~\ref{lem:mulE}.
\end{proof}

\begin{lemma}\label{lem:noME}
Given an edge $\eta\in\mathcal{T}$ and $\Lambda_{\eta}$, $H_{\eta}$ as above, let $p$, $q$ be the indices of the two endpoints of $\eta$, respectively. Then, $\Lambda_{\eta}$ has no multiple edges if $p$ and $q$ are coprime, or if $p$ is prime and $H_{\eta}$ is not cyclic.
\end{lemma}

\begin{proof}
One can get $H_1$, $H_2$ as above and assume that $|H_1|=p$, $|H_2|=q$. If $p$ and $q$ are coprime, then $H_0$ is trivial by Lemma~\ref{lem:theHi}. By Lemma~\ref{lem:mulE}, we get the result. If $H_{\eta}$ is not cyclic, then $H_{\eta}\neq H_2$. By Lemma~\ref{lem:theHi}, we have $H_0\neq H_1$. Then, $H_0$ is trivial when $p$ is prime, and we get the result as before.
\end{proof}


\subsection{Large actions and hyperbolic triangle groups}\label{subsec:LAandHTG}
Now we mainly consider those actions $(\Sigma_g,G)$ with $g\geq 2$ and $|G|>12(g-1)$. We assume that $|\Sigma_g/G|$ has genus $\bar{g}$ and there exist $l$ singular points with indices $n_1,\ldots,n_l$ in $\Sigma_g/G$. Then by the Riemann-Hurwitz formula, we have
\[-\frac{1}{6}<2-2\bar{g}-\sum_{k=1}^{l}(1-\frac{1}{n_k})=\frac{2-2g}{|G|}<0.\]
So $\bar{g}=0$ and $l=3$. Hence in what follows we can assume that $\Sigma_g/G$ is a $2$-sphere with three singular points $P$, $Q$, $R$ which have indices $p\leq q\leq r$, respectively, and there is a pseudo-triangulation $\mathcal{T}$ of $\Sigma_g/G$ with exactly two pseudo-triangles. The pseudo-triangulation is actually unique up to isotopy.

\begin{proposition}\label{pro:cases}
Let $\Sigma_g/G$ be as above, where $|G|>12(g-1)>0$. Then

(1) $(p,q,r)=(2,3,r)$, where $r\geq 7$; or

(2) $(p,q,r)=(2,4,r)$, where $5\leq r\leq 11$; or

(3) $(p,q,r)=(2,5,r)$, where $5\leq r\leq 7$; or

(4) $(p,q,r)=(3,3,r)$, where $4\leq r\leq 5$.
\end{proposition}

\begin{proof}
Now the Riemann-Hurwitz formula and the above inequality become
\[\frac{5}{6}<\frac{1}{p}+\frac{1}{q}+\frac{1}{r}=\frac{2-2g}{|G|}+1<1.\]
Then one can obtain the result by enumeration.
\end{proof}

\begin{proposition}\label{pro:ETforLA}
Let $\Sigma_g/G$ and $\mathcal{T}$ be as above, where $|G|>12(g-1)>0$. The map $\phi:\pi_1(\Sigma_g/G,\overline{\star})\rightarrow G$ defines $\mathcal{T}_{\phi}$ as before. Then this $\mathcal{T}_{\phi}$ gives a $G$-equivariant triangulation of $\Sigma_g$ except when $g=2$ and $(p,q,r)$ is $(2,4,8)$.
\end{proposition}

\begin{proof}
Let $\eta\in\mathcal{T}$ be an edge. By the presentation of $\pi_1(\Sigma_g/G,\overline{\star})$ and the proof of Lemma~\ref{lem:theHi}, we have $H_{\eta}=\phi(\pi_1(N,\overline{\star}))=G$. So the preimage of $\eta$ in $\Sigma_g$ is $\Lambda_{\eta}$. We also see that $H_{\eta}$ is not cyclic, for otherwise, $12(g-1)<|G|\leq 4g+2$ (see \cite{Wi}).

Then, by Proposition~\ref{pro:cases} and Lemma~\ref{lem:noME}, we see that $\Lambda_{\eta}$ can contain multiple edges only if $(p,q,r)$ is one of $(2,4,6)$, $(2,4,8)$, or $(2,4,10)$, and $\eta$ is the edge with endpoints $Q$ and $R$. Furthermore, by Lemma~\ref{lem:mulE}, there exists a nontrivial element $h\in H_0$ which also lies in the center of $G$. So $|\mathrm{Fix}(h)|\geq |G|/q+|G|/r>|G|/3$. By Lemma~\ref{lem:fixN} below, this is possible only if $g=2$, since $|G|>12(g-1)$.

If $g=2$, then by the Riemann-Hurwitz formula, $(p,q,r)$ cannot be $(2,4,10)$. In the case when $(p,q,r)$ is $(2,4,6)$, by Remark~\ref{rem:TwoE} below, the action is unique, and $H_0$ is trivial. So the only possible case is when $(p,q,r)$ is $(2,4,8)$. Since $\mathcal{T}_{\phi}$ has no multiple edges in the other cases, we get the result by Lemma~\ref{lem:EofET}.
\end{proof}

\begin{lemma}\label{lem:fixN}
Let $h$ be a nontrivial periodic map on $\Sigma_g$. Then $|\mathrm{Fix}(h)|\leq 2g+2$.
\end{lemma}

\begin{proof}
Assume that $h$ has order $m$, $|\mathrm{Fix}(h)|=l$, and the orbifold $\Sigma_g/\langle h\rangle$ has $k+l$ singular points with indices $n_1,\ldots,n_{k+l}$, where $n_{k+j}=m$ for all $1\leq j\leq l$. We let $\bar{g}$ be the genus of $|\Sigma_g/\langle h\rangle|$. Then by the Riemann-Hurwitz formula, we have
\[2-2g=m(2-2\bar{g}-\sum_{j=1}^{k}(1-\frac{1}{n_j})-l(1-\frac{1}{m}))\leq 2m-l(m-1).\]
So $l\leq 2g/(m-1)+2\leq 2g+2$, and we get the result.
\end{proof}

\begin{remark}\label{rem:TwoE}
Those actions $(\Sigma_g,G)$ with $g=2$ have been classified in \cite{Br}. There exists exactly one action if $(p,q,r)$ is $(2,4,6)$ or $(2,4,8)$, respectively. For the case of $(2,4,6)$, $\mathcal{T}_{\phi}$ has no multiple edges, but for $(2,4,8)$, it has multiple edges.

Let $(p,q,r)=(2,4,r)$, where $r=6$ or $r=8$. Then $\pi_1(\Sigma_g/G,\overline{\star})$ is given by
\[\langle \alpha,\beta,\gamma \mid \alpha^2=\beta^4=\gamma^r=\alpha\beta\gamma=1 \rangle.\]

If $r=6$, then the group $G$ and the map $\phi:\pi_1(\Sigma_g/G,\overline{\star})\rightarrow G$ are given by
\begin{align*}
\langle x,y,z,w \mid x^2=y^2&=z^2=w^3=[y,z]=[y,w]=[z,w]=1,\\
xyx^{-1}&=y, xzx^{-1}=zy, xwx^{-1}=w^{-1}\rangle
\end{align*}
and $\alpha\mapsto x$, $\beta\mapsto xw^2z$, $\gamma\mapsto zw$, respectively. Note that $\phi(\beta)^2=y$, $G$ is actually a semidirect product $\mathbb{Z}_2\ltimes(\mathbb{Z}_2\times\mathbb{Z}_2\times\mathbb{Z}_3)$, and $y$ generates the center of $G$. Because $\phi(\gamma)^3=z$, we see that $H_0=\langle\phi(\beta)\rangle\cap\langle\phi(\gamma)\rangle$ is trivial.

If $r=8$, then the group $G$ and the map $\phi:\pi_1(\Sigma_g/G,\overline{\star})\rightarrow G$ are given by
\begin{align*}
\langle x,y \mid x^2=y^8=1, xyx^{-1}=y^{3}\rangle
\end{align*}
and $\alpha\mapsto x$, $\beta\mapsto y^5x$, $\gamma\mapsto y$, respectively. Now $G$ is a semidirect product $\mathbb{Z}_2\ltimes\mathbb{Z}_8$, $\phi(\beta)^2=y^4$ generates the center of $G$, and $H_0=\langle\phi(\beta)\rangle\cap\langle\phi(\gamma)\rangle=\langle y^4\rangle$.
\end{remark}

By Proposition~\ref{pro:ETforLA} and Remark~\ref{rem:tri}, for most large actions, we already have
\[d_g(G)\leq \frac{|G|}{p}+\frac{|G|}{q}+\frac{|G|}{r} <|G|.\]
To obtain better upper bounds of $d_g(G)$, we also need to modify the triangulation given by Proposition~\ref{pro:ETforLA}, and we need to apply Theorem~\ref{thm:tri}.

A convenient way to get the required triangulation of $\Sigma_g$ from $\mathcal{T}_{\phi}$ is to consider the construction of hyperbolic triangle groups. We let $\mathbb{H}^2$ be the hyperbolic plane, and let $\mathrm{Isom}(\mathbb{H}^2)$ be the isometry group of $\mathbb{H}^2$. For any positive integers $p\leq q\leq r$ satisfying $1/p+1/q+1/r<1$, there exists a unique geodesic triangle $\Delta_{p,q,r}$ in $\mathbb{H}^2$ with angles $\pi/p$, $\pi/q$, and $\pi/r$, up to isometry. Then the reflections in those three geodesic lines containing the edges of $\Delta_{p,q,r}$ generate a discrete group in $\mathrm{Isom}(\mathbb{H}^2)$ where the orbit of $\Delta_{p,q,r}$ under the group action gives a tessellation of $\mathbb{H}^2$. And in this group, the orientation-preserving elements give an index two subgroup, which is the hyperbolic triangle group with type $(p,q,r)$, denoted by $T_{p,q,r}$. So one has a pseudo-triangulation of the orbifold $\mathbb{H}^2/T_{p,q,r}$, which contains exactly two pseudo-triangles coming from the geodesic triangles in $\mathbb{H}^2$.

Since we can identify $\Sigma_g/G$ with $\mathbb{H}^2/T_{p,q,r}$ for some $(p,q,r)$, we can view $\Sigma_g$ as a hyperbolic surface, where the triangles in $\mathcal{T}_{\phi}$ are all geodesic ones. Note that
\[\pi_1(\Sigma_g/G,\overline{\star})\cong \langle \alpha,\beta,\gamma \mid \alpha^p=\beta^q=\gamma^r=\alpha\beta\gamma=1 \rangle \cong T_{p,q,r},\]
where $\alpha$, $\beta$, and $\gamma$ correspond to the rotations about the three vertices of $\Delta_{p,q,r}$ in $T_{p,q,r}$, with angles $2\pi/p$, $2\pi/q$, and $2\pi/r$, respectively. Then, by identifying $\ker(\phi)$ with a subgroup of $T_{p,q,r}$, we can identify $\Sigma_g$ with $\mathbb{H}^2/\ker(\phi)$. Figure~\ref{fig:TriG2} shows two examples corresponding to the two cases in Remark~\ref{rem:TwoE}, where those edges of the polygons that should be identified by elements in $\ker(\phi)$ are labelled.

\begin{figure}[h]
\includegraphics{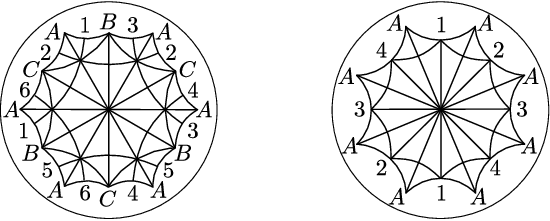}
\caption{The $(2,4,6)$ and $(2,4,8)$ triangular tilings of $\Sigma_2$.}\label{fig:TriG2}
\end{figure}

Now as in Proposition~\ref{pro:ETforLA}, assume that $\mathcal{T}_{\phi}$ gives a $G$-equivariant triangulation of $\Sigma_g$, then we modify $\mathcal{T}_{\phi}$ as follows. Let $\eta\in\mathcal{T}$ be the edge with endpoints $P$ and $Q$, and let $\widetilde{\eta}$ be an edge in $\Lambda_{\eta}$. Then we have two geodesic triangles $\Delta_1$ and $\Delta_2$ in $\mathcal{T}_{\phi}$ containing $\widetilde{\eta}$. There are two cases according to Proposition~\ref{pro:cases}. If $p=2$, then $\Delta_1\cup\Delta_2$ gives a larger geodesic triangle, since $\mathcal{T}_{\phi}$ has no multiple edges. Then, all such larger triangles induce a new pseudo-triangulation of $\Sigma_g$. Otherwise, we have $p>2$, and $\Delta_1\cup\Delta_2$ gives a convex geodesic quadrilateral by the same reason. The diagonal other than $\widetilde{\eta}$ divides it into two geodesic triangles. So, similarly, all these triangles induce a new pseudo-triangulation of $\Sigma_g$. Note that any new triangle has two vertices lying in the preimage of $R$. So we denote the pseudo-triangulation by $\mathcal{T}_{\phi}(R)$. Similarly, by interchanging $Q$ and $R$, we can also define $\mathcal{T}_{\phi}(Q)$. We let $V_P$, $V_Q$, and $V_R$ denote the preimages of $P$, $Q$, and $R$ in $\Sigma_g$, respectively.

\begin{proposition}\label{pro:IndET}
Let $\Sigma_g/G$, $\mathcal{T}$, and $\mathcal{T}_{\phi}$ be as above, where $|G|>12(g-1)>0$. Then, the induced $\mathcal{T}_{\phi}(Q)$ or $\mathcal{T}_{\phi}(R)$ gives a $G$-equivariant triangulation of $\Sigma_g$ if the index $q$ or $r$ is prime, respectively. Moreover, if $(p,q,r)$ is $(2,3,8)$ and $g>2$, then $\mathcal{T}_{\phi}(R)$ also gives a $G$-equivariant triangulation of $\Sigma_g$.
\end{proposition}

\begin{proof}
We only prove the case of $\mathcal{T}_{\phi}(R)$. The case of $\mathcal{T}_{\phi}(Q)$ is similar.

By the construction of $\mathcal{T}_{\phi}(R)$ and Lemma~\ref{lem:PTtoT}, we only need to show that it has no multiple edges. Then it is easy to construct a $G$-equivariant triangulation from $\mathcal{T}_{\phi}(R)$ by using the hyperbolic geometry of $\Sigma_g$.

Assume that there exists a pair of multiple edges $\eta$ and $\eta'$ in $\mathcal{T}_{\phi}(R)$. Then their common vertices lie in $V_R$, since $\mathcal{T}_{\phi}$ has no multiple edges. Note that all the edges in $\mathcal{T}_{\phi}(R)$ with endpoints lying in $V_R$ form a connected graph. Similar to the proof of Lemma~\ref{lem:mulE}, we see that there exists an element $h$ in $\mathrm{Stab}(\partial\eta)$ so that $h(\eta)=\eta'$ and $h$ fixes all the vertices in $V_R$. However, this cannot happen if $r$ is prime, since by considering the edges adjacent to a vertex in $V_Q$, we have $|V_R|\geq q>2$.

If $(p,q,r)$ is $(2,3,8)$, then by Lemma~\ref{lem:fixN} and the Riemann-Hurwitz formula,
\[2g+2\geq |\mathrm{Fix}(h)|\geq |V_R|=\frac{|G|}{8}=\frac{48(g-1)}{8}.\]
Clearly this cannot happen if $g>2$. So we get the result.
\end{proof}

\begin{lemma}\label{lem:Upoint}
Let $(K,\tau)$ be a triangulation of $\Sigma_g$, and let $U\subset K^0$ be a subset so that $\mathrm{St}(v)\cap\mathrm{St}(w)=\mathrm{Lk}(v)\cap\mathrm{Lk}(w)$ for any $v\neq w$ in $U$, and $|\mathrm{Lk}(u)^0|\leq 4$ for each $u$ in $U$. Then $\mathrm{Lk}(v)^0\neq\mathrm{Lk}(w)^0$ for any $v\neq w$ in $U$, unless $g=0$ and $|U|=2$.
\end{lemma}

\begin{proof}
Assume that $\mathrm{Lk}(v)^0=\mathrm{Lk}(w)^0$ for some $v\neq w$ in $U$. And it has $s$ elements. Then $s=3$ or $s=4$. If $s=3$, then $|\mathrm{St}(v)|\cup|\mathrm{St}(w)|=|K|$. So, we get $|U|=2$ and $g=0$. If $s=4$, then $|\mathrm{St}(v)|\cup|\mathrm{St}(w)|$ is either $|K|$ or a M\"obius band. And we also get $|U|=2$ and $g=0$, since $\Sigma_g$ is orientable.
\end{proof}

\begin{proof}[Proof of Theorem~\ref{thm:SLBound}]
By Remark~\ref{rem:smallG}, $d_g(G)<12$ if $|G|\leq 3$. This inequality also holds if $|G|=4$, by using Theorem~\ref{thm:exist}. So one can assume that $|G|\geq 5$. Then, by Theorem~\ref{thm:GEBound}, one can further assume that $|G|>12(g-1)$. So there are four cases below, according to Proposition~\ref{pro:cases}. And in each case we will apply Theorem~\ref{thm:tri} to some suitable triangulations of $\Sigma_g$ to get the upper bounds.

Case 1: $(p,q,r)=(3,3,r)$, where $4\leq r\leq 5$. If $r=4$, then, by Proposition~\ref{pro:IndET}, $\mathcal{T}_{\phi}(Q)$ gives a $G$-equivariant triangulation $(K,\tau)$ of $\Sigma_g$. Now let $U=V_P\cup V_R$ and $V=V_Q$. By Lemma~\ref{lem:Upoint}, $K^0=U\cup V$ satisfies the conditions in Theorem~\ref{thm:tri}. So we have $d_g(G)\leq |V_Q|$. By the Riemann-Hurwitz formula, we see that
\[d_g(G)\leq |V_Q|=\frac{|G|}{3}=\frac{24(g-1)}{3}<12(g-1).\]
Similarly, if $r=5$, then $\mathcal{T}_{\phi}(R)$ gives a $G$-equivariant triangulation of $\Sigma_g$, where we let $U=V_P\cup V_Q$ and $V=V_R$. By Theorem~\ref{thm:tri}, we have
\[d_g(G)\leq |V_R|=\frac{|G|}{5}=\frac{15(g-1)}{5}<12(g-1).\]

Case 2: $(p,q,r)=(2,5,r)$, where $5\leq r\leq 7$. Now, by Proposition~\ref{pro:ETforLA}, $\mathcal{T}_{\phi}$ gives a $G$-equivariant triangulation $(K,\tau)$ of $\Sigma_g$. And we let $U=V_P$ and $V=V_Q\cup V_R$. Then, by Lemma~\ref{lem:Upoint}, $K^0=U\cup V$ satisfies the conditions in Theorem~\ref{thm:tri}. So
\[d_g(G)\leq |V_Q|+|V_R|=(\frac{1}{5}+\frac{1}{r})|G|=\frac{r+5}{5r}\times\frac{20r(g-1)}{3r-10}<12(g-1).\]

Case 3: $(p,q,r)=(2,4,r)$, where $5\leq r\leq 11$. If $r$ is $5$, $7$, or $11$, then $\mathcal{T}_{\phi}(R)$ will give a $G$-equivariant triangulation of $\Sigma_g$, by Proposition~\ref{pro:IndET}. Now we let $U=V_Q$ and $V=V_R$. By Lemma~\ref{lem:Upoint} and Theorem~\ref{thm:tri}, we have
\[d_g(G)\leq |V_R|=\frac{|G|}{r}=\frac{1}{r}\times\frac{8r(g-1)}{r-4}<12(g-1).\]
Otherwise, $r\geq 6$. By Proposition~\ref{pro:ETforLA}, if it is not the case when $(p,q,r)=(2,4,8)$ and $g=2$, then $\mathcal{T}_{\phi}$ gives a $G$-equivariant triangulation of $\Sigma_g$. As in Case~2, we let $U=V_P$ and $V=V_Q\cup V_R$. By Lemma~\ref{lem:Upoint} and Theorem~\ref{thm:tri},
\[d_g(G)\leq |V_Q|+|V_R|=(\frac{1}{4}+\frac{1}{r})|G|=\frac{r+4}{4r}\times\frac{8r(g-1)}{r-4}<12(g-1).\]
If $r=8$ and $g=2$, then there exists only one action, by Remark~\ref{rem:TwoE}. In this case the corresponding hyperbolic surface $\Sigma_g=\mathbb{H}^2/\ker(\phi)$ actually has isometry group $T_{2,3,8}/\ker(\phi)$, and by Remark~\ref{rem:238Act} below, the inequality still holds.

Case 4: $(p,q,r)=(2,3,r)$, where $r\geq 7$. If $r$ is prime, then by Proposition~\ref{pro:IndET}, $\mathcal{T}_{\phi}(R)$ gives a $G$-equivariant triangulation of $\Sigma_g$. As in Case~3, we let $U=V_Q$ and $V=V_R$. Then, by Lemma~\ref{lem:Upoint} and Theorem~\ref{thm:tri}, we have
\[d_g(G)\leq |V_R|=\frac{|G|}{r}=\frac{1}{r}\times\frac{12r(g-1)}{r-6}\leq 12(g-1).\]
The same argument also holds when $r=8$ and $g>2$. If $r=8$ and $g=2$, then, by Remark~\ref{rem:238Act} below, the action is unique, and the inequality still holds. Hence one can assume that $r\geq 9$. Then consider the $G$-equivariant triangulation of $\Sigma_g$ given by $\mathcal{T}_{\phi}(Q)$. Let $U=V_R$ and $V=V_Q$. If the conditions in Theorem~\ref{thm:tri} hold, then
\[d_g(G)\leq |V_Q|=\frac{|G|}{3}=\frac{1}{3}\times\frac{12r(g-1)}{r-6}\leq 12(g-1).\]
Otherwise, we can obtain a contradiction as follows.

Note that $\mathrm{St}(u)\cap\mathrm{St}(u')=\mathrm{Lk}(u)\cap\mathrm{Lk}(u')$ for any $u\neq u'$ in $U$, and $|\mathrm{St}(v)|$ gives a disk as in the left picture of Figure~\ref{fig:3uT} for each $v\in V$, where $u_1,u_2,u_3\in U$. If we have $u\neq u'$ in $U$ so that $\mathrm{Lk}(u)^0=\mathrm{Lk}(u')^0$, then $\mathrm{St}(u)$ and $\mathrm{St}(u')$ share a common edge, as in the picture, and there are $2$ choices of $u'$ for a fixed $u$. Let $h\in\mathrm{Stab}(u)$ be a $2\pi/r$-rotation about $u$. Then $h^2(u')=u'$, and we can see that $|U|=3$. So let $U=\{u_1,u_2,u_3\}$. The Riemann-Hurwitz formula implies that $r=4g+2$.

\begin{figure}[h]
\includegraphics{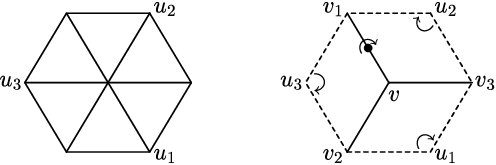}
\caption{The disk $|\mathrm{St}(v)|$ and the rotations about $u_j$.}\label{fig:3uT}
\end{figure}

Now, as in the right picture of Figure~\ref{fig:3uT}, the $\pi$-rotation interchanging $u_2$ and $u_3$ must give a $\pi$-rotation of $|\mathrm{St}(u_1)|$. Then we see that $v$ and $v_j$ are antipodal points in $|\mathrm{St}(u_j)|$ for $1\leq j\leq 3$. Let $h_j\in\mathrm{Stab}(u_j)$ denote the $2\pi/(2g+1)$-rotation about $u_j$ shown in the picture. Since $h_3(v_1)=v_2$ and $U\subseteq\mathrm{Fix}(h_3)$, we get $h_3=h_2^g=h_1^g$. Similarly, $h_2=h_1^g$, and so $g^2\equiv g \pmod{2g+1}$, which is impossible.
\end{proof}

\begin{remark}\label{rem:238Act}
By the result in \cite{Br} there exists exactly one action when $g=2$ and $(p,q,r)=(2,3,8)$. The group $G$ and the map $\phi:\pi_1(\Sigma_g/G,\overline{\star})\rightarrow G$ are given by
\[
\mathrm{GL}(2,\mathbb{Z}/3\mathbb{Z})=\langle x,y \mid x=
\begin{pmatrix}
1 & 1 \\
0 & -1 \\
\end{pmatrix}, y=
\begin{pmatrix}
0 & -1 \\
1 & -1 \\
\end{pmatrix}\rangle
\]
and $\alpha\mapsto x$, $\beta\mapsto y$, $\gamma\mapsto y^2x$, respectively. The hyperbolic surface $\Sigma_g=\mathbb{H}^2/\ker(\phi)$ is the same one in the case of $(2,4,8)$. The $(2,3,8)$-triangles can be obtained from those ones in the left picture of Figure~\ref{fig:238LB} by taking the barycentric subdivision. So we see that $\mathcal{T}_{\phi}(R)$ has multiple edges. Now $\mathcal{T}_{\phi}(Q)$ can give a triangulation so that the conditions in Theorem~\ref{thm:tri} hold, where $U=V_R$ and $V=V_Q$. However, we can only get $d_g(G)\leq |V|=16$ by the theorem, which is not good enough. In the right picture of Figure~\ref{fig:238LB},
we label the vertices in $V_Q$ by nonzero vectors in $(\mathbb{Z}/3\mathbb{Z})^2$ and signs in $\{+,-\}$. Then we can show that $d_g(G)\leq 8$ as follows.

\begin{figure}[h]
\includegraphics{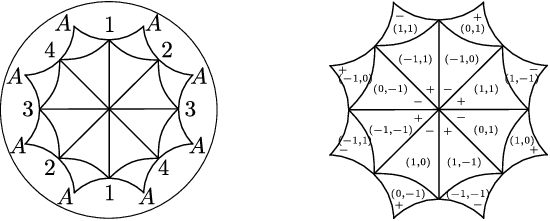}
\caption{The surface $\Sigma_2$ and the labels of vertices in $V_Q$.}\label{fig:238LB}
\end{figure}

Note that vertices in $V_Q$ are the barycenters of the triangles. By identifying $V_Q$ with $\Omega\times\{+,-\}$, where $\Omega=(\mathbb{Z}/3\mathbb{Z})^2\setminus\{(0,0)\}$, according to the picture, it is easy to check that the action of $\mathrm{GL}(2,\mathbb{Z}/3\mathbb{Z})$ on $V_Q$ can be given by
\begin{align*}
A: \Omega\times\{+,-\} & \rightarrow\Omega\times\{+,-\} \\
(\omega, \delta) & \mapsto (A\omega, \mathrm{sign}(\det A)\delta)
\end{align*}
for $A\in\mathrm{GL}(2,\mathbb{Z}/3\mathbb{Z})$, where the $\omega$ is written as column vectors. For example, now the element $y$ fixes $((1,-1),+)$ and gives a $2\pi/3$-rotation near the vertex. We can write $\Omega=\{\omega_1,\ldots,\omega_8\}$ and let $\hat{e}((\omega_j,+))=e_j$ and $\hat{e}((\omega_j,-))=-e_j$ for $1\leq j\leq 8$, then let $\hat{e}(u)=\sum\hat{e}(v)$ for each $u\in U$, where the sum is over all those $v\in\mathrm{Lk}(u)^0$, then as in the proof of Theorem~\ref{thm:tri}, we can get an embedding $\hat{e}:\Sigma_2\rightarrow\mathbb{R}^8$, which is $\mathrm{GL}(2,\mathbb{Z}/3\mathbb{Z})$-equivariant. By similar arguments as in the case of Theorem~\ref{thm:tri}, $\hat{e}$ can be replaced by a required smooth embedding. On the other hand, we can also verify $d_g(G)\leq 8$ by using Theorem~\ref{thm:exist}, based on the existence of $\hat{e}$.
\end{remark}

\begin{remark}
In the proof of Theorem~\ref{thm:SLBound}, we see that in many cases those upper bounds of $d_g(G)$ can be much smaller than $12(g-1)$. Actually, one can have even better upper bounds under certain conditions, by choosing suitable $U$-tessellations of $\Sigma_g$. The whole method can also be applied to those pairs $(\Sigma_g,G)$ related to the hyperbolic triangle groups with general types $(p,q,r)$.
\end{remark}


\section{The Riemann surfaces from principal congruence subgroups}\label{sec:RSfromPCS}
In Section~\ref{subsec:CCofET}, we will give a precise definition of $(\Sigma(p),\overline{\Gamma}_p)$ and the equivariant triangulation of $\Sigma(p)$ derived from the $(2,3,p)$ triangular tiling. In Section~\ref{subsec:EEinCR}, we will use the equivariant triangulation to construct smooth equivariant embeddings from $\Sigma(p)$ to $\mathbb{C}^{p+1}$ and $\mathbb{R}^{p+1}$. Then we will prove Theorem~\ref{thm:PCS} in Section~\ref{subsec:MEED}.

\subsection{Combinatorial construction of the equivariant triangulation}\label{subsec:CCofET}
We can identify the upper half-plane $\{z\in\mathbb{C}\mid\mathrm{Im}(z)>0\}$ with $\mathbb{H}^2$ for simplicity. Given an element $S$ in $\mathrm{SL}(2,\mathbb{Z})$, its image in $\mathrm{PSL}(2,\mathbb{Z})$ will be denoted by $\overline{S}$. And its image in $\overline{\Gamma}_p$ will be denoted by $\overline{S}_p$. The action of $\overline{S}$ on $\mathbb{H}^2$ comes from the map
\[S=
\begin{pmatrix}
a & b \\
c & d \\
\end{pmatrix}\mapsto f(z)=\frac{az+b}{cz+d}.
\]
It is well known that the $\mathrm{PSL}(2,\mathbb{Z})$-action on $\mathbb{H}^2$ gives an orbifold. It has an open disk as the underlying space. And it has two singular points, with indices $2$ and $3$, respectively. So its fundamental group is $\mathrm{PSL}(2,\mathbb{Z})\cong\mathbb{Z}/2\mathbb{Z}\ast\mathbb{Z}/3\mathbb{Z}$.

Let $I$ denote the identity matrix in $\mathrm{SL}(2,\mathbb{Z})$, and let $A$, $B$, $C$ be given by
\[
A=
\begin{pmatrix}
0 & -1 \\
1 & 0 \\
\end{pmatrix}, B=
\begin{pmatrix}
0 & 1 \\
-1 & 1 \\
\end{pmatrix}, C=
\begin{pmatrix}
1 & 1 \\
0 & 1 \\
\end{pmatrix}.
\]
Then $\overline{A}$ and $\overline{B}$ are those generators of $\mathrm{PSL}(2,\mathbb{Z})$ with orders $2$ and $3$, respectively. And so $\overline{A}_p$ and $\overline{B}_p$ generate $\overline{\Gamma}_p\cong\mathrm{PSL}(2,\mathbb{Z}/p\mathbb{Z})$. Note that $ABC=I$ and
\[C^p=
\begin{pmatrix}
1 & p \\
0 & 1 \\
\end{pmatrix}\equiv
\begin{pmatrix}
1 & 0 \\
0 & 1 \\
\end{pmatrix}\pmod{p}.
\]
Hence we have the following epimorphism which is injective on finite subgroups
\[\phi:\langle\alpha,\beta,\gamma \mid \alpha^2=\beta^3=\gamma^p=\alpha\beta\gamma=1\rangle\rightarrow\overline{\Gamma}_p,\]
where $\alpha\mapsto\overline{A}_p$, $\beta\mapsto\overline{B}_p$, $\gamma\mapsto\overline{C}_p$, respectively. This gives a covering space $\Sigma(p)$ of the orbifold $\mathbb{H}^2/T_{2,3,p}$, which is a closed Riemann surface with genus
\[g=\frac{1}{2}(1-\frac{1}{2}-\frac{1}{3}-\frac{1}{p})|\mathrm{PSL}(2,\mathbb{Z}/p\mathbb{Z})|+1= \frac{1}{24}(p^2-1)(p-6)+1.\]
Its automorphism group is isomorphic to $\overline{\Gamma}_p\cong\mathrm{PSL}(2,\mathbb{Z}/p\mathbb{Z})$.

Similarly, the epimorphism $\phi_0:\mathrm{PSL}(2,\mathbb{Z})=\langle\overline{A},\overline{B}\rangle\rightarrow\overline{\Gamma}_p$ gives a covering space of $\mathbb{H}^2/\mathrm{PSL}(2,\mathbb{Z})$, which is just $\mathbb{H}^2/\overline{\Gamma}(p)$. By comparing $\phi_0$ with $\phi$, we can see that topologically $\mathbb{H}^2/\overline{\Gamma}(p)$ is the same as $\Sigma(p)$ with $(p^2-1)/2$ points removed.

Now, as in Section~\ref{sec:ETandPT}, we have a pseudo-triangulation $\mathcal{T}_{\phi}$ of $\Sigma(p)$ which gives an equivariant triangulation. And the triangles in $\mathcal{T}_{\phi}$ are all geodesic ones. Note that every point in the preimage of the singular point in $\mathbb{H}^2/T_{2,3,p}$ of index $3$ lies in six triangles, whose union is an equilateral triangle with angle $2\pi/p$, since $\mathcal{T}_{\phi}$ contains no multiple edges. So we get a new pseudo-triangulation of $\Sigma(p)$, which is induced by all such equilateral triangles. And we denote it by $\mathcal{T}_p$. Since $p$ is prime, $\mathcal{T}_p$ also gives an equivariant triangulation of $\Sigma(p)$ as before. On the other hand, there is a combinatorial construction of this new triangulation as follows.

By identifying each element $(x,y)$ in $(\mathbb{Z}/p\mathbb{Z})^2\setminus\{(0,0)\}$ with $(-x,-y)$, we get a set $V_p$ so that $|V_p|=(p^2-1)/2$. We let $\overline{(x,y)}$ denote the image of $(x,y)$ in $V_p$ and we regard it as a vertex. We add an edge between the vertices $\overline{(x,y)}$ and $\overline{(x',y')}$ if and only if $xy'-x'y$ is $1$ or $-1$. If there are three edges connecting three vertices, then we further add a triangle. This gives a simplicial complex $K$.

One can check that $|K|$ is connected, and if there is an edge between $\overline{(x,y)}$ and $\overline{(x',y')}$, then there are only two vertices adjacent to both of them, which are given by $\overline{(x+x',y+y')}$ and $\overline{(x-x',y-y')}$, respectively. And around each vertex there are exactly $p$ triangles. So one can see that $|K|$ is a connected closed surface. The group $\overline{\Gamma}_p\cong\mathrm{PSL}(2,\mathbb{Z}/p\mathbb{Z})$ naturally acts on $V_p$ by the following equation
\[
\overline{S}_p\cdot
\overline{
\begin{pmatrix}
x \\
y \\
\end{pmatrix}}=
\overline{
\begin{pmatrix}
\overline{a} & \overline{b} \\
\overline{c} & \overline{d} \\
\end{pmatrix}
\begin{pmatrix}
x \\
y \\
\end{pmatrix}}=
\overline{
\begin{pmatrix}
\overline{a}x+\overline{b}y \\
\overline{c}x+\overline{d}y \\
\end{pmatrix}}
\]
where the image of an integer $m$ in $\mathbb{Z}/p\mathbb{Z}$ is denoted by $\overline{m}$. Since the edges can be regarded as matrices, with those vertices as columns, it is not hard to see that the $\overline{\Gamma}_p$-action on $V_p$ extends to $|K|$ via simplicial maps.

\begin{proposition}\label{pro:ETforSp}
There exists a homeomorphism $\tau:|K|\rightarrow\Sigma(p)$ so that $(|K|,\overline{\Gamma}_p)$ is conjugate to $(\Sigma(p),\overline{\Gamma}_p)$, while $K$ corresponds to $\mathcal{T}_p$, via $\tau$.
\end{proposition}

\begin{proof}
We first show that the orbifolds $|K|/\overline{\Gamma}_p$ and $\mathbb{H}^2/T_{2,3,p}$ are the same. Let $\Delta$ be a triangle in $K$. One can assume that $\Delta$ has $\overline{(x,y)}$, $\overline{(x',y')}$, and $\overline{(x+x',y+y')}$ as its vertices, where $xy'-x'y=1$. There is also a triangle $\Delta_0$ in $K$ with vertices $\overline{(1,0)}$, $\overline{(0,1)}$, and $\overline{(1,1)}$. Let $S_p$ denote the matrix given by $x$, $x'$, $y$, $y'$. Since
\[
\begin{pmatrix}
x & x' \\
y & y' \\
\end{pmatrix}
\begin{pmatrix}
1 \\
0 \\
\end{pmatrix}=
\begin{pmatrix}
x \\
y \\
\end{pmatrix},
\begin{pmatrix}
x & x' \\
y & y' \\
\end{pmatrix}
\begin{pmatrix}
0 \\
1 \\
\end{pmatrix}=
\begin{pmatrix}
x' \\
y' \\
\end{pmatrix},
\]
the image of $S_p$ in $\mathrm{PSL}(2,\mathbb{Z}/p\mathbb{Z})$ gives an element $\overline{S}_p\in\overline{\Gamma}_p$ so that $\overline{S}_p(\Delta_0)=\Delta$. If we write $\overline{(0,1)}$ as $\overline{(0,-1)}$, then we see that $\overline{B}_p(\Delta_0)=\Delta_0$ and it is a $2\pi/3$-rotation on $\Delta_0$. Note that there are exactly $p$ elements in $\overline{\Gamma}_p$ which can fix $\overline{(1,0)}$, and they are rotations near $\overline{(1,0)}$. So only those powers of $\overline{B}_p$ keep $\Delta_0$ invariant. By taking the barycentric subdivision of $\Delta_0$, we see that the union of the two small triangles containing $\overline{(1,0)}$ is a fundamental domain, whose quotient gives $\mathbb{H}^2/T_{2,3,p}$.

So $|K|$ is orientable. By computing the Euler characteristic we see that $|K|$ has the same genus as $\Sigma(p)$. Now we construct a concrete map $\overline{\tau}:|K|/\overline{\Gamma}_p\rightarrow\mathbb{H}^2/T_{2,3,p}$ as follows. Note that $\overline{A}_p$ gives a $\pi$-rotation near the midpoint of the edge between $\overline{(1,0)}$ and $\overline{(0,1)}$. Let $\Delta_{\overline{\star}}\subset|K|/\overline{\Gamma}_p$ be the image of the small triangle that contains $\overline{(1,0)}$ and the midpoint. It has a basepoint $\overline{\star}$ lying in its interior. One can assume that $\Delta_{\overline{\ast}}\subset\mathbb{H}^2/T_{2,3,p}$ is the pseudo-triangle which has a basepoint $\overline{\ast}$ in the interior. Then we can have an embedding $\overline{\tau}$ so that $\overline{\tau}(\Delta_{\overline{\star}})=\Delta_{\overline{\ast}}$ and $\overline{\tau}(\overline{\star})=\overline{\ast}$. We see that essentially the map $\pi_1(|K|/\overline{\Gamma}_p,\overline{\star})\rightarrow\overline{\Gamma}_p$ is the same as $\phi$. Then by Lemma~\ref{lem:lift} and Remark~\ref{rem:lift}, the lift of $\overline{\tau}$ gives the required $\tau$.
\end{proof}

Figure~\ref{fig:TriG3} shows an example of the construction, where $p=7$ and elements in $V_7$ are just written as pairs of integers for simplicity. It is well known that $\Sigma(7)$ gives the Klein quartic, which is a Riemann surface of genus $3$. Similar to the two cases shown in Figure~\ref{fig:TriG2}, $\Sigma(7)$ can also be obtained from a regular $14$-gon by identifying certain pairs of its edges, as shown in the figure below.

\begin{figure}[h]
\includegraphics{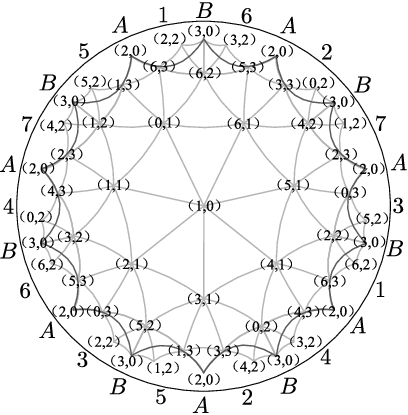}
\caption{The tiling of $\Sigma(7)$ given by $\mathcal{T}_7$ with labels in $V_7$.}\label{fig:TriG3}
\end{figure}

\begin{remark}\label{rem:Farey}
Similarly, by identifying $(s,t)$ with $(-s,-t)$ in the set of all the pairs of coprime integers, we have a set $V$, where the image of $(s,t)$ is denoted by $\overline{(s,t)}$. We regard $V$ as a vertex set and add an edge between $\overline{(s,t)}$ and $\overline{(s',t')}$ if and only if $st'-s't=\pm 1$. Then, by adding the triangles as before, we can have a simplicial complex $K^\ast$. Now each edge is still contained in exactly two triangles, but around each vertex those triangles form an infinite sequence. By the Euclidean algorithm, we can see that $|K^\ast|$ is connected. And the action of $\mathrm{PSL}(2,\mathbb{Z})$ on $V$ given by
\[
\overline{S}\cdot
\overline{
\begin{pmatrix}
s \\
t \\
\end{pmatrix}}=
\overline{
\begin{pmatrix}
a & b \\
c & d \\
\end{pmatrix}
\begin{pmatrix}
s \\
t \\
\end{pmatrix}}=
\overline{
\begin{pmatrix}
as+bt \\
cs+dt \\
\end{pmatrix}}
\]
extends to $|K^\ast|$ via simplicial maps. Similar to Proposition~\ref{pro:ETforSp}, one can show that the actions of $\mathrm{PSL}(2,\mathbb{Z})$ on $|K^\ast|\setminus V$ and $\mathbb{H}^2$ are essentially the same. Actually, by identifying $\overline{(s,t)}$ with $s/t\in\mathbb{Q}\cup\{\infty\}$, the $1$-skeleton in $|K^\ast|$ naturally corresponds to the famous Farey diagram which gives a $\mathrm{PSL}(2,\mathbb{Z})$-invariant ideal triangulation of $\mathbb{H}^2$. Then we can see that the simplicial map $|K^\ast|\rightarrow|K|$ induced by $\mathbb{Z}\rightarrow\mathbb{Z}/p\mathbb{Z}$ naturally corresponds to the covering map $\mathbb{H}^2\rightarrow\mathbb{H}^2/\overline{\Gamma}(p)$.
\end{remark}


\subsection{Equivariant embeddings in the complex and real spaces}\label{subsec:EEinCR}
Below those elements $\overline{(x,y)}$ in $V_p$ will be denoted by $[s:t]$ for simplicity, where $s,t\in\mathbb{Z}$ so that $\overline{s}=x$, $\overline{t}=y$. Hence we have $[s:t]=[s+p:t]=[s:t+p]$, $[s:t]=[-s:-t]$. Now let $r$ be a primitive root modulo $p$. So $\{\overline{r},\ldots,\overline{r}^{p-1}\}=\{\overline{1},\ldots,\overline{p-1}\}$, and we have $r^{(p-1)/2}\equiv -1\pmod{p}$. Then, we can list the elements in $V_p$ as follows. There are $p+1$ columns, corresponding to different ratios, and $(p-1)/2$ rows.
\[\begin{matrix}
[1:0] & [0:1] & [1:1] & \cdots & [p-1:1] \\
[r:0] & [0:r] & [r:r] & \cdots & [(p-1)r:r] \\
\vdots & \vdots & \vdots &  & \vdots \\
[r^{\frac{p-3}{2}}:0] & [0:r^{\frac{p-3}{2}}] & [r^{\frac{p-3}{2}}:r^{\frac{p-3}{2}}] & \cdots & [(p-1)r^{\frac{p-3}{2}}:r^{\frac{p-3}{2}}]\\
\end{matrix}\]
By the definition of the $\overline{\Gamma}_p$-action on $V_p$, for $\overline{S}_p\in\overline{\Gamma}_p$ and $[sr^q:tr^q]\in V_p$, we have
\[\overline{S}_p([sr^q:tr^q])=[(as+bt)r^q:(cs+dt)r^q].\]
So $\overline{S}_p$ permutes those $p+1$ columns and it permutes the elements in each column cyclically.
We can use the list to first construct an embedding $\hat{e}:|K|\rightarrow\mathbb{C}^{p+1}$.

We let $\{\eta_{\infty},\eta_0,\eta_1,\ldots,\eta_{p-1}\}$ denote the standard basis for $\mathbb{C}^{p+1}$, corresponding to the $p+1$ columns, and let $\omega=\mathrm{e}^{4\pi \mathrm{i}/(p-1)}$. Then, we let
$\hat{e}([sr^q:r^q])=\omega^q\eta_s$ and $\hat{e}([r^q:0])=\omega^q\eta_{\infty}$, where $0\leq s\leq p-1$ and $0\leq q\leq (p-3)/2$. This defines a map $V_p\rightarrow\mathbb{C}^{p+1}$, and its linear extension gives the map $\hat{e}$. Note that there are no edges between the vertices lying in the same column. So, it is clear that $\hat{e}$ is injective on any triangle in $|K|$, and as in the proof of Theorem~\ref{thm:tri}, we see that the embedded triangles in $\mathbb{C}^{p+1}$ only meet in common edges or common vertices. Then since $\hat{e}$ is injective on $V_p$, it is also injective on $|K|$. So $\hat{e}$ is an embedding.

Now by regarding $\mathbb{C}^{p+1}$ as $\mathbb{R}^{2p+2}$ we define a representation $\rho: \overline{\Gamma}_p\rightarrow\mathrm{O}(2p+2)$ as follows. For each $1\leq s\leq p-1$, let $s^\ast$ and $k_s$ denote the two unique integers so that $ss^\ast\equiv 1\pmod{p}$, $r^{k_s}\equiv s\pmod{p}$, and $1\leq s^\ast, k_s\leq p-1$. Then, let
\[\rho(\overline{A}_p)(\eta_{\infty})=\eta_0,\, \rho(\overline{A}_p)(\eta_0)=\eta_{\infty},\, \rho(\overline{A}_p)(\eta_s)=\omega^{k_s}\eta_{p-s^\ast},\, 1\leq s\leq p-1.\]
The relations provide a complex linear map $\rho(\overline{A}_p)$, which lies in $\mathrm{O}(2p+2)$. Since
\[\overline{A}_p([sr^q:r^q])=[-r^q:sr^q]=[-ss^{\ast}r^q:sr^q]=[(p-s^\ast)r^{q+k_s}:r^{q+k_s}],\]
it is easy to check that on $V_p$ we have $\hat{e}\circ\overline{A}_p=\rho(\overline{A}_p)\circ\hat{e}$. Similarly, let
\[\rho(\overline{C}_p)(\eta_{\infty})=\eta_{\infty},\, \rho(\overline{C}_p)(\eta_{p-1})=\eta_0,\, \rho(\overline{C}_p)(\eta_s)=\eta_{s+1},\, 0\leq s\leq p-2.\]
This gives an element $\rho(\overline{C}_p)\in\mathrm{O}(2p+2)$ satisfying $\hat{e}\circ\overline{C}_p=\rho(\overline{C}_p)\circ\hat{e}$ on $V_p$. Note that $\overline{A}_p$ and $\overline{C}_p$ generate $\overline{\Gamma}_p$, and the action of $\overline{\Gamma}_p$ on $V_p$ is faithful. So we have an embedding $\rho$. And $\hat{e}$ is a $\overline{\Gamma}_p$-equivariant embedding with respect to $\rho$.

\begin{proposition}\label{pro:EEofSp}
Let $\hat{e}$ and $\rho$ be as above. Then there is a smooth $\overline{\Gamma}_p$-equivariant embedding $e:\Sigma(p)\rightarrow\mathbb{C}^{p+1}$ with respect to $\rho$ so that $e(\Sigma(p))$ lies near $\hat{e}(|K|)$. Also, there exists a smooth $\overline{\Gamma}_p$-equivariant embedding $e_0:\Sigma(p)\rightarrow\mathbb{R}^{p+1}$.
\end{proposition}

\begin{proof}
By Proposition~\ref{pro:ETforSp}, we can identify $\Sigma(p)$ with $|K|$ via $\tau$. As in the proof of Theorem~\ref{thm:tri}, we need to modify $\hat{e}$ to give the required $e$. Then by Lemma~\ref{lem:lift}, we can require that those points in $\Sigma(p)$ where $\hat{e}$ is not smooth correspond to vertices and edges of the triangles in $\hat{e}(|K|)$, and it suffices to show that $\hat{e}$ can be modified so that it is smooth near the singular points in $\hat{e}(|K^1|)$. Let $y$ be such a point and let $\rho(h)$ be a generator of $\mathrm{Stab}(y)$, where $h\in\overline{\Gamma}_p$. Now there are two cases.

If $y$ is the midpoint of an edge, then one can assume that this edge has $\eta_{\infty}$ and $\eta_0$ as the endpoints, and $h$ is $\overline{A}_p$. Then, the other two vertices of the two triangles containing $y$ are $\eta_1$ and $\eta_{p-1}$, which are also interchanged by $\rho(\overline{A}_p)$. Hence we see that the method in Case~1 in the proof of Theorem~\ref{thm:tri} works. Otherwise, one can assume that $y$ is the vertex $\eta_{\infty}$, and $h$ is $\overline{C}_p$. In this case, the other vertices of the triangles containing $y$ are $\eta_0,\ldots,\eta_{p-1}$, which form one $\rho(\overline{C}_p)$-orbit with $p\geq 7$. So we see that the method in Case~2 in the proof of Theorem~\ref{thm:tri} works. Hence $\hat{e}$ can be modified so that it is smooth near $\hat{e}^{-1}(y)$, and we can get the required $e$.

To get the required $e_0$, we will apply Theorem~\ref{thm:exist}. First note that essentially $\rho$ is a $(p+1)$-dimensional complex representation, and the representation is actually conjugate to a real representation $\rho_0:\overline{\Gamma}_p\rightarrow\mathrm{O}(p+1)$. In Appendix~\ref{app:RR} we will give a proof of this fact and write down $\rho_0$ explicitly. So the two real representations $\rho$ and the direct sum $\rho_0\oplus\rho_0$ are conjugate. Then, with this property, one can check the three conditions in Theorem~\ref{thm:exist} for $\rho_0$ as follows.

Since $e:\Sigma(p)\rightarrow\mathbb{C}^{p+1}$ is a smooth $\overline{\Gamma}_p$-equivariant embedding with respect to $\rho$, by Lemma~\ref{lem:nec}, $\rho$ satisfies the dimension inequality and eigenvalue condition. Note that for each element $\overline{S}_p$ in $\overline{\Gamma}_p$, the eigenvalues of $\rho(\overline{S}_p)$, with multiplicity, are two copies of the eigenvalues of $\rho_0(\overline{S}_p)$, and for each subgroup $H$ of $\overline{\Gamma}_p$, we have
\[\mathrm{dim}\mathrm{Fix}(\rho(H))=2\mathrm{dim}\mathrm{Fix}(\rho_0(H)).\]
Clearly $\rho_0$ also satisfies the dimension inequality. Since $\rho(\overline{A}_p)$ interchanges each of $\{\eta_{\infty},\eta_0\}$ and $\{\eta_1,\eta_{p-1}\}$, we see that $\rho_0(\overline{A}_p)$ has $-1$ as a multiple eigenvalue. And then it is easy to see that $\rho_0$ satisfies the eigenvalue condition as well.

Since $p\geq 7$, to check the codimension condition, we only need to show that the inequality $\mathrm{dim}\mathrm{Fix}(\rho_0(\overline{S}_p))<p-1$ holds for any nontrivial $\overline{S}_p\in\overline{\Gamma}_p$. Let $\rho_{\mathbb{C}}$ denote the complex representation corresponding to $\rho$. Then, $\rho_{\mathbb{C}}$ induces permutations on $\{\infty,0,1,\ldots,p-1\}$, and so gives a new representation $\rho_{\mathbb{R}}:\overline{\Gamma}_p\rightarrow\mathrm{O}(p+1)$. And
\[\mathrm{dim}_{\mathbb{C}}\mathrm{Fix}(\rho_{\mathbb{C}}(h))= \frac{1}{m}\sum_{j=1}^m\mathrm{tr}(\rho_{\mathbb{C}}(h^j))\leq \frac{1}{m}\sum_{j=1}^m\mathrm{tr}(\rho_{\mathbb{R}}(h^j))=l,\]
where $h\in\overline{\Gamma}_p$ has order $m$, and the permutation induced by $\rho_{\mathbb{C}}(h)$ has $l$ orbits. So we have $\mathrm{dim}\mathrm{Fix}(\rho_0(h))=\mathrm{dim}_{\mathbb{C}}\mathrm{Fix}(\rho_{\mathbb{C}}(h))\leq l$.

If $\mathrm{dim}\mathrm{Fix}(\rho_0(\overline{S}_p))\geq p-1$, then the corresponding $l$ is $p+1$, $p$, or $p-1$. Hence the permutation induced by $\rho_{\mathbb{C}}(\overline{S}_p)$ has four possibilities, up to conjugation. Since now $\mathrm{dim}_{\mathbb{C}}\mathrm{Fix}(\rho_{\mathbb{C}}(\overline{S}_p))\geq p-1$, it is not hard to see that $\rho_{\mathbb{C}}(\overline{S}_p)$ fixes at least $p-3$ elements in $\{\eta_{\infty},\eta_0,\eta_1,\ldots,\eta_{p-1}\}$. Then since $p\geq 7$, $\overline{S}_p$ fixes an edge in $|K|$. So it is the identity, and the codimension condition holds.
\end{proof}

\begin{remark}\label{rem:explic}
Note that we have written down $\rho_{\mathbb{C}}$ explicitly. Hence it is possible to verify conditions in Theorem~\ref{thm:exist} and prove Proposition~\ref{pro:EEofSp} directly. However, this will need a lot of computations. Also, after we giving $\rho_0$ explicitly in Appendix~\ref{app:RR}, it is possible to provide a more precise construction of the required $e_0$.
\end{remark}


\subsection{The minimal equivariant embedding dimension}\label{subsec:MEED}
Now, by investigating certain properties of the element $\overline{C}_p\in\overline{\Gamma}_p\cong\mathrm{PSL}(2,\mathbb{Z}/p\mathbb{Z})$, we can finish the proof of Theorem~\ref{thm:PCS}. First note that $\overline{C}_p$ fixes exactly $(p-1)/2$ points in $V_p$, namely
\[[1:0], [2:0], \ldots, [\frac{p-1}{2}:0].\]
It is easy to see that there are exactly $p$ vertices adjacent to $[s:0]$ in $K$, given by
\[[0:s^\ast], [s:s^\ast], \ldots, [(p-1)s:s^\ast].\]
By definition, the action of $\overline{C}_p$ on these $p$ points is given by
\[\overline{C}_p([js:s^\ast])=[js+s^\ast:s^\ast]=[(j+(s^\ast)^2)s:s^\ast], \,0\leq j\leq p-1.\]
So the action of $\overline{C}_p$ on $|K|$ gives a $2\pi(s^\ast)^2/p$-rotation near $[s:0]$.

\begin{proposition}\label{pro:4k3}
If $p\equiv 3\pmod{4}$, then $d_g(G)\geq p+1$ for the pair $(\Sigma(p),\overline{\Gamma}_p)$.
\end{proposition}

\begin{proof}
Suppose that there is a $\overline{\Gamma}_p$-equivariant embedding $\Sigma(p)\rightarrow\mathbb{R}^n$, with respect to an embedding $\rho:\overline{\Gamma}_p\rightarrow\mathrm{O}(n)$. Then, by Proposition~\ref{pro:ETforSp} and Lemma~\ref{lem:nec}, we see that $\rho(\overline{C}_p)$ has eigenvalues $\mathrm{e}^{\pm2\pi(s^\ast)^2\mathrm{i}/p}$ for each $1\leq s\leq (p-1)/2$. These numbers are all different, since $p\equiv 3\pmod{4}$. Since those $(p-1)/2\geq 3$ fixed points of $\overline{C}_p$ lie in the same orbit of the $\overline{\Gamma}_p$-action, $\mathrm{dim}\mathrm{Fix}(\rho(\overline{C}_p))\geq 2$. So $n\geq p+1$.
\end{proof}

The above argument does not hold when $p\equiv 1\pmod{4}$, because there are only $(p-1)/2$ different eigenvalues. In this case, we need a little representation theory. It is well known that the dimension of any irreducible complex representation $\rho$ of $\mathrm{PSL}(2,\mathbb{Z}/p\mathbb{Z})$ is $1$, $(p+1)/2$, $p-1$, $p$, or $p+1$, where $p\equiv 1\pmod{4}$. The powers of $\overline{C}_p$ other than $\overline{I}_p$ lie in two conjugacy classes. Each class contains half of them, and has the representative $\overline{C}_p$ or $(\overline{C}_p)^r$, where $r$ is the primitive root modulo $p$. If $\rho$ has dimension $(p+1)/2$, then $\mathrm{tr}(\rho(\overline{C}_p))+\mathrm{tr}(\rho((\overline{C}_p)^r))=1$. For the other cases, the value of $\mathrm{tr}(\rho(\overline{C}_p))+\mathrm{tr}(\rho((\overline{C}_p)^r))$ is $2$, $-2$, or $0$. See \cite[$\S 5.2$]{FH}.

\begin{proposition}\label{pro:4k1}
If $p\equiv 1\pmod{4}$, then $d_g(G)\geq p+1$ for the pair $(\Sigma(p),\overline{\Gamma}_p)$.
\end{proposition}

\begin{proof}
Suppose that there is a $\overline{\Gamma}_p$-equivariant embedding $\Sigma(p)\rightarrow\mathbb{R}^n$, with respect to an embedding $\rho:\overline{\Gamma}_p\rightarrow\mathrm{O}(n)$. Let $k=\mathrm{tr}(\rho(\overline{C}_p))+\mathrm{tr}(\rho((\overline{C}_p)^r))$. Then
\[\mathrm{dim}\mathrm{Fix}(\rho(\overline{C}_p)) =\frac{1}{p}\sum_{j=1}^p\mathrm{tr}(\rho((\overline{C}_p)^j))=\frac{1}{p}(\frac{(p-1)k}{2}+n).\]
Hence $2p\cdot\mathrm{codim}\mathrm{Fix}(\rho(\overline{C}_p))=(p-1)(2n-k)$. Also, $\mathrm{dim}\mathrm{Fix}(\rho(\overline{C}_p))\geq 2$ as before. Now we regard $\rho$ as a complex representation. So, one can write $\rho$ as a direct sum of irreducible representations. Assume that $n\leq p$.

If no summand of $\rho$ has dimension $(p+1)/2$, then $k\in\mathbb{Z}$ is an even number. So $\mathrm{codim}\mathrm{Fix}(\rho(\overline{C}_p))=p-1$. And $n\geq (p-1)+2$, which contradicts the assumption. Hence we have $\rho=\rho'\oplus\rho_1$, where $\rho'$ is a summand of dimension $(p+1)/2$, and $\rho_1$ is a direct sum of trivial representations. Then we also have such a decomposition where all the representations are regarded as real ones. Let $\mathbb{R}^n=W'\oplus W_1$ be the corresponding decomposition. Let $\xi$ be a fixed point of $\rho(\overline{C}_p)$. So $\xi=\xi'+\xi_1$ with $\xi'\in W'$ and $\xi_1\in W_1$. Then $\rho(h)(\xi)=\rho'(h)(\xi')+\xi_1$ for each $h\in\overline{\Gamma}_p$. Hence $\xi'$ is a fixed point of $\rho'(\overline{C}_p)$, and all such fixed points coming from the fixed points of $\overline{C}_p$ lie in the same orbit of the $\rho'(\overline{\Gamma}_p)$-action. On the other hand, we have
\[\mathrm{dim}\mathrm{Fix}(\rho'(\overline{C}_p)) =\frac{1}{p}\sum_{j=1}^p\mathrm{tr}(\rho'((\overline{C}_p)^j))=\frac{1}{p}(\frac{p-1}{2}+\frac{p+1}{2})=1.\]
Since there are $(p-1)/2\geq 3$ fixed points, we get a contradiction.

Hence we must have $n\geq p+1$.
\end{proof}

\begin{proof}[Proof of Theorem~\ref{thm:PCS}]
By Propositions~\ref{pro:EEofSp}, \ref{pro:4k3}, and \ref{pro:4k1}, we have the result.
\end{proof}

\begin{remark}\label{rem:icosa}
For the case $p=5$, one can have similar constructions, where $\Sigma(5)$ is just the Riemann sphere, $\mathcal{T}_5$ can be given by twenty spherical equilateral triangles with angle $2\pi/5$, and $\overline{\Gamma}_5$ is essentially the isometry group of a regular icosahedron. Clearly $d_g(G)=3$ for $(\Sigma(5),\overline{\Gamma}_5)$. So the representation $\rho$ has dimension $(5+1)/2$, and it is irreducible. However, different from the cases when $p\geq 7$, $\overline{C}_5$ has $2$ fixed points, which can be embedded in a line with the equal norm.
\end{remark}


\begin{appendix}
\section{A concrete real representation of $\mathrm{PSL}(2,\mathbb{Z}/p\mathbb{Z})$}\label{app:RR}
As mentioned in the proof of Proposition~\ref{pro:EEofSp}, in this section, we will prove that the complex representation $\rho_{\mathbb{C}}$ is conjugate to a real representation $\rho_0$. By using a little representation theory, it is not hard to see that $\rho_{\mathbb{C}}$ is irreducible, since $p\geq 7$. Actually, for any given $\overline{S}_p\in\overline{\Gamma}_p$, one can obtain $\rho_{\mathbb{C}}(\overline{S}_p)$ by rewriting each of
\[\overline{S}_p([1:0])=[a:c],\, \overline{S}_p([s:1])=[as+b:cs+d],\, 0\leq s\leq p-1,\]
as one of $[r^q:0]$, $[sr^q:r^q]$, $0\leq s\leq p-1$, for some $0\leq q\leq (p-3)/2$, and one can then get $\mathrm{tr}(\rho_{\mathbb{C}}(\overline{S}_p))$ and verify that $\rho_{\mathbb{C}}$ lies in the character table of $\mathrm{PSL}(2,\mathbb{Z}/p\mathbb{Z})$. Then one can check that the second Frobenius-Schur indicator of $\rho_{\mathbb{C}}$ satisfies
\[\nu_2(\rho_{\mathbb{C}})=\frac{1}{|G|}\sum_{h\in G}\mathrm{tr}(\rho_{\mathbb{C}}(h^2))=1,\]
where $G=\mathrm{PSL}(2,\mathbb{Z}/p\mathbb{Z})$. This gives a way to show that $\rho_{\mathbb{C}}$ is real. See \cite[$\S 3.5$]{FH}. Below we give another way, and we will write down the $\rho_0$ explicitly.

Let $(a_{i,j})$ and $(c_{i,j})$ be the matrices of $\rho_{\mathbb{C}}(\overline{A}_p)$ and $\rho_{\mathbb{C}}(\overline{C}_p)$, respectively, and let $(\delta_{i,j})$ be the identity matrix, where $i$ and $j$ lie in $\{\infty,0,1,\ldots,p-1\}$. We have
\[a_{i,\infty}=\delta_{i,0},\, a_{i,0}=\delta_{i,\infty},\, a_{i,s}=\omega^{k_s}\delta_{i,p-s^\ast},\, 1\leq s\leq p-1,\]
where $s^\ast$ and $k_s$ satisfy $ss^\ast\equiv 1\pmod{p}$, $r^{k_s}\equiv s\pmod{p}$, and $1\leq s^\ast, k_s\leq p-1$, and $\omega=\mathrm{e}^{4\pi \mathrm{i}/(p-1)}$. It is easy to see that $(a_{i,j})^t=\overline{(a_{i,j})}$. Similarly, we have
\[c_{i,\infty}=\delta_{i,\infty},\, c_{i,p-1}=\delta_{i,0},\, c_{i,s}=\delta_{i,s+1},\, 0\leq s\leq p-2.\]
Then, since $\rho_{\mathbb{C}}(\overline{\Gamma}_p)$ is generated by $\rho_{\mathbb{C}}(\overline{A}_p)$ and $\rho_{\mathbb{C}}(\overline{C}_p)$, it suffices to find a matrix $(z_{i,j})$ so that $(z_{i,j})^{-1}(a_{i,j})(z_{i,j})$ and $(z_{i,j})^{-1}(c_{i,j})(z_{i,j})$ have real entries.

Let $\varepsilon=\mathrm{e}^{2\pi \mathrm{i}/p}$, and let $(y_{i,j})$ be the matrix defined by
\[y_{\infty,\infty}=1,\, y_{i,\infty}=y_{\infty,j}=0,\, y_{i,j}=\frac{\varepsilon^{ij}}{\sqrt{p}},\, 0\leq i,j\leq p-1.\]
Then $(y_{i,j})^t=(y_{i,j})$ and $(y_{i,j})^{-1}=\overline{(y_{i,j})}$. Let $(A_{i,j})=(y_{i,j})^{-1}(a_{i,j})(y_{i,j})$. Then
\begin{align*}
A_{\infty,\infty}&=\sum_{l,s}\overline{y_{\infty,l}}\cdot a_{l,s}\cdot y_{s,\infty}=a_{\infty,\infty}=0,\\
A_{i,\infty}&=\sum_{l,s}\overline{y_{i,l}}\cdot a_{l,s}\cdot y_{s,\infty} =\sum_{l}\overline{y_{i,l}}\cdot a_{l,\infty}=\overline{y_{i,0}}=\frac{1}{\sqrt{p}},\\
A_{i,j}&=\sum_{l,s}\overline{y_{i,l}}\cdot a_{l,s}\cdot y_{s,j} =\sum_{l,s\neq\infty}\overline{y_{i,l}}\cdot a_{l,s}\cdot y_{s,j}
=\sum_{l,s\neq\infty,0}\overline{y_{i,l}}\cdot a_{l,s}\cdot y_{s,j}\\
&=\frac{1}{p}\sum_{l,s\neq\infty,0}\varepsilon^{-il}\cdot \omega^{k_s}\delta_{l,p-s^\ast}\cdot \varepsilon^{sj}=\frac{1}{p}\sum_{s=1}^{p-1}\omega^{k_s}\varepsilon^{is^\ast+sj},
\end{align*}
where $0\leq i,j\leq p-1$, and $(A_{i,j})^t=\overline{(A_{i,j})}$. Let $(C_{i,j})=(y_{i,j})^{-1}(c_{i,j})(y_{i,j})$. Then
\begin{align*}
C_{\infty,\infty}&=c_{\infty,\infty}=1,\, C_{i,\infty}=\overline{y_{i,\infty}}=0,\, C_{\infty,j}=y_{\infty,j}=0,\\
C_{i,j}&=\sum_{l,s}\overline{y_{i,l}}\cdot c_{l,s}\cdot y_{s,j} =\sum_{l,s\neq\infty}\overline{y_{i,l}}\cdot c_{l,s}\cdot y_{s,j}\\
&=\frac{1}{p}\sum_{l,s\neq\infty}\varepsilon^{-il+sj}c_{l,s} =\frac{1}{p}\sum_{s=0}^{p-1}\varepsilon^{(j-i)s-i}=\varepsilon^{-i}\delta_{i,j},
\end{align*}
where $0\leq i,j\leq p-1$, and clearly $(C_{i,j})$ is a diagonal matrix.

\begin{lemma}\label{lem:Aij}
The entries $A_{i,j}$, $0\leq i,j\leq p-1$ satisfy the following properties

(1) $A_{0,0}=0$ and $|A_{1,0}|=|A_{0,1}|=1/\sqrt{p}$;

(2) $A_{i,j}=\omega^{k_i}A_{1,n}$ if $i\neq 0$ and $A_{i,j}=\omega^{-k_j}A_{n,1}$ if $j\neq 0$, where $ij\equiv n\pmod{p}$ and $0\leq n\leq p-1$, and moreover $A_{i,0}=A_{p-i,0}$, $A_{i,j}=A_{p-i,p-j}$ for $i,j\neq 0$.
\end{lemma}

\begin{proof}
(1) Since $\omega=\mathrm{e}^{4\pi \mathrm{i}/(p-1)}$ and $\varepsilon=\mathrm{e}^{2\pi i/p}$, $A_{0,0}=(\sum_{s=1}^{p-1}\omega^{k_s})/p=0$, and
\begin{align*}
|A_{1,0}|^2&=|A_{0,1}|^2\\
&=\frac{1}{p^2}(\sum_{s=1}^{p-1}\omega^{k_s}\varepsilon^s) (\sum_{l=1}^{p-1}\omega^{-k_l}\varepsilon^{-l})
=\frac{1}{p^2}(\sum_{s=1}^{p-1}\omega^{s}\varepsilon^{r^s}) (\sum_{l=1}^{p-1}\omega^{-l}\varepsilon^{-r^{l}})\\
&=\frac{1}{p^2}\sum_{s=1}^{p-1}\sum_{l=1}^{p-1}\omega^{s-l}\varepsilon^{r^s-r^l} =\frac{1}{p^2}\sum_{m=1}^{p-1}\sum_{l=1}^{p-1}\omega^{m}\varepsilon^{r^{m+l}-r^l}\\
&=\frac{1}{p^2}(\sum_{m=1}^{p-2}\sum_{l=1}^{p-1}\omega^{m}\varepsilon^{(r^m-1)r^l}+ \sum_{l=1}^{p-1}\varepsilon^{(r^{p-1}-1)r^l})\\
&=\frac{1}{p^2}(-\sum_{m=1}^{p-2}\omega^{m}+p-1)=\frac{1}{p}.
\end{align*}

(2) For $i\neq 0$, since $p-i\equiv r^{(p-1)/2+k_i}\pmod{p}$, we have $\omega^{k_{p-i}}=\omega^{k_i}$. Then we see that $A_{i,j}=\omega^{k_i}A_{1,n}$ implies $A_{i,j}=\omega^{k_{p-i}}A_{1,n}=A_{p-i,p-j}$, $j\neq 0$, $A_{i,0}=A_{p-i,0}$, and $A_{i,j}=\overline{A_{j,i}}=\omega^{-k_j}\overline{A_{1,n}}=\omega^{-k_j}A_{n,1}$ for $j\neq 0$.

To finish the proof, for $i\neq 0$, we have
\begin{align*}
A_{i,j}&=\overline{A_{j,i}}\\
&=\frac{1}{p}\sum_{s=1}^{p-1}\omega^{-k_s}\varepsilon^{-js^\ast-si} =\frac{\omega^{k_i}}{p}\sum_{s=1}^{p-1}\omega^{-k_i-k_s}\varepsilon^{-ji(si)^\ast-si}\\
&=\frac{\omega^{k_i}}{p}\sum_{l=1}^{p-1}\omega^{-k_l}\varepsilon^{-nl^\ast-l}=\omega^{k_i}A_{1,n},
\end{align*}
where we have used the relation $\omega^{k_l}=\omega^{k_s+k_i}$ for $l\equiv si\pmod{p}$.
\end{proof}

Now choose $\lambda$ and $\mu$ such that $\lambda^2=\sqrt{p}A_{1,0}$ and $\mu^2=\omega$. Let $q=(p-1)/2$ and let $\lambda_j=\lambda\mu^{k_j}$ for $1\leq j\leq q$. Then $|\lambda_j|=1$. Let $(x_{i,j})$ be the matrix defined by
\begin{align*}
x_{i,\infty}&=\frac{\delta_{i,\infty}+\delta_{i,0}}{\sqrt{2}},\, x_{i,0}=\frac{\delta_{i,\infty}-\delta_{i,0}}{\mathrm{i}\sqrt{2}},\\
x_{i,j}&=\frac{(\delta_{i,j}+\delta_{i,p-j})\lambda_j}{\sqrt{2}},\, x_{i,p-j}=\frac{(\delta_{i,j}-\delta_{i,p-j})\lambda_j}{\mathrm{i}\sqrt{2}},
\end{align*}
where $i$ lies in $\{\infty,0,1,\ldots,p-1\}$ and $1\leq j\leq q$. Then $(x_{i,j})^{-1}=\overline{(x_{i,j})^t}$.

Below we verify that $(x_{i,j})^{-1}(A_{i,j})(x_{i,j})$ and $(x_{i,j})^{-1}(C_{i,j})(x_{i,j})$ lie in $\mathrm{O}(p+1)$. We first compute the $(i,j)$-entries of $(x_{i,j})^{-1}(A_{i,j})(x_{i,j})$ with $j=\infty$ and $0\leq j\leq i$ and show that they are real numbers.

For $(i,j)=(\infty,\infty),(0,\infty),(0,0)$, we have
\begin{align*}
\sum_{l,s}\overline{x_{l,\infty}}\cdot A_{l,s}\cdot x_{s,\infty}&= \frac{1}{2}\sum_{l,s}(\delta_{l,\infty}+\delta_{l,0})(\delta_{s,\infty}+\delta_{s,0})A_{l,s} =\frac{1}{\sqrt{p}},\\
\sum_{l,s}\overline{x_{l,0}}\cdot A_{l,s}\cdot x_{s,\infty}&= -\frac{1}{2\mathrm{i}}\sum_{l,s}(\delta_{l,\infty}-\delta_{l,0})(\delta_{s,\infty}+\delta_{s,0})A_{l,s} =0,\\
\sum_{l,s}\overline{x_{l,0}}\cdot A_{l,s}\cdot x_{s,0}&= \frac{1}{2}\sum_{l,s}(\delta_{l,\infty}-\delta_{l,0})(\delta_{s,\infty}-\delta_{s,0})A_{l,s} =-\frac{1}{\sqrt{p}}.
\end{align*}

The $(j,\infty)$-entries and $(j,0)$-entries, where $1\leq j\leq q$, are given by
\begin{align*}
\sum_{l,s}\overline{x_{l,j}}\cdot A_{l,s}\cdot x_{s,\infty}&= \frac{\overline{\lambda_j}}{2} \sum_{l,s}(\delta_{l,j}+\delta_{l,p-j})(\delta_{s,\infty}+\delta_{s,0})A_{l,s} =\overline{\lambda_j}(\frac{1}{\sqrt{p}}+A_{j,0}),\\
\sum_{l,s}\overline{x_{l,j}}\cdot A_{l,s}\cdot x_{s,0}&= \frac{\overline{\lambda_j}}{2\mathrm{i}} \sum_{l,s}(\delta_{l,j}+\delta_{l,p-j})(\delta_{s,\infty}-\delta_{s,0})A_{l,s} =\frac{\overline{\lambda_j}}{\mathrm{i}}(\frac{1}{\sqrt{p}}-A_{j,0}).
\end{align*}
Here we have used the property $A_{j,0}=A_{p-j,0}$ in Lemma~\ref{lem:Aij}. Since $A_{1,0}=\lambda^2/\sqrt{p}$ and $\lambda^2\omega^{k_j}=\lambda^2\mu^{2k_j}=\lambda_j^2$, by Lemma~\ref{lem:Aij} we obtain $A_{j,0}=\omega^{k_j}A_{1,0}=\lambda_j^2/\sqrt{p}$, and so the above $2q$ entries are real. Similarly, by using $A_{j,0}=A_{p-j,0}$, one can get the following $(p-j,\infty)$-entries and $(p-j,0)$-entries, where $1\leq j\leq q$,
\begin{align*}
\sum_{l,s}\overline{x_{l,p-j}}\cdot A_{l,s}\cdot x_{s,\infty}&= -\frac{\overline{\lambda_j}}{2\mathrm{i}} \sum_{l,s}(\delta_{l,j}-\delta_{l,p-j})(\delta_{s,\infty}+\delta_{s,0})A_{l,s}=0,\\
\sum_{l,s}\overline{x_{l,p-j}}\cdot A_{l,s}\cdot x_{s,0}&= \frac{\overline{\lambda_j}}{2} \sum_{l,s}(\delta_{l,j}-\delta_{l,p-j})(\delta_{s,\infty}-\delta_{s,0})A_{l,s}=0.
\end{align*}

Finally, we have the following remaining entries, where $1\leq i,j\leq q$,
\begin{align*}
\sum_{l,s}\overline{x_{l,i}}\cdot A_{l,s}\cdot x_{s,j}&= \frac{\overline{\lambda_i}\lambda_j}{2} \sum_{l,s}(\delta_{l,i}+\delta_{l,p-i})(\delta_{s,j}+\delta_{s,p-j})A_{l,s}\\ &=\overline{\lambda_i}\lambda_j(A_{i,j}+A_{i,p-j})=\mu^{k_j-k_i}(A_{i,j}+A_{i,p-j}),\\
\sum_{l,s}\overline{x_{l,p-i}}\cdot A_{l,s}\cdot x_{s,j}&= -\frac{\overline{\lambda_i}\lambda_j}{2\mathrm{i}} \sum_{l,s}(\delta_{l,i}-\delta_{l,p-i})(\delta_{s,j}+\delta_{s,p-j})A_{l,s}=0,\\
\sum_{l,s}\overline{x_{l,p-i}}\cdot A_{l,s}\cdot x_{s,p-j}&= \frac{\overline{\lambda_i}\lambda_j}{2} \sum_{l,s}(\delta_{l,i}-\delta_{l,p-i})(\delta_{s,j}-\delta_{s,p-j})A_{l,s}\\ &=\overline{\lambda_i}\lambda_j(A_{i,j}-A_{i,p-j})=\mu^{k_j-k_i}(A_{i,j}-A_{i,p-j}).
\end{align*}
Here we have applied $A_{i,j}=A_{p-i,p-j}$ and $A_{i,p-j}=A_{p-i,j}$ in Lemma~\ref{lem:Aij}. And by the lemma, we also have $\overline{A_{i,j}}=A_{j,i}=\omega^{k_j-k_i}A_{i,j}$. So $\mu^{k_j-k_i}A_{i,j}$ is real. Similarly, $\mu^{k_j-k_i}A_{i,p-j}$ is real. Hence the above entries are all real numbers.

Then, since $(A_{i,j})^t=\overline{(A_{i,j})}$, $(x_{i,j})^{-1}=\overline{(x_{i,j})^t}$, we see that $(x_{i,j})^{-1}(A_{i,j})(x_{i,j})$ is a real symmetric matrix. So it lies in $\mathrm{O}(p+1)$. Similarly, one can compute all the $(i,j)$-entries of $(x_{i,j})^{-1}(C_{i,j})(x_{i,j})$, denoted by $\theta_{i,j}$, and check that
\begin{align*}
\theta_{s,\infty}&=\theta_{\infty,s}=\delta_{s,\infty},\, \theta_{s,0}=\theta_{0,s}=\delta_{s,0},\\
\theta_{i,j}&=\theta_{p-i,p-j}=\frac{(\varepsilon^{i}+\varepsilon^{-i})\delta_{i,j}}{2},\, \theta_{p-i,j}=-\theta_{i,p-j}=\frac{(\varepsilon^{i}-\varepsilon^{-i})\delta_{i,j}}{2\mathrm{i}},
\end{align*}
where $s$ lies in $\{\infty,0,1,\ldots,p-1\}$ and $1\leq i,j\leq q$. It is easy to see that $(\theta_{i,j})$ lies in $\mathrm{O}(p+1)$ as well. Hence one can choose $(z_{i,j})=(y_{i,j})(x_{i,j})$.

\begin{remark}
For the case $p=5$, one can have similar computations. However, we note that the representation $\rho_{\mathbb{C}}$ in this case is already a real one, because $\omega=-1$, and the corresponding matrices $(z_{i,j})^{-1}(a_{i,j})(z_{i,j})$ and $(z_{i,j})^{-1}(c_{i,j})(z_{i,j})$ will show directly how it decomposes into two irreducible ones of dimension $3$.
\end{remark}

Now with the concrete $(z_{i,j})$ and $\rho_0$, one can have a concrete $e_0:\Sigma(p)\rightarrow\mathbb{R}^{p+1}$, as mentioned in Remark~\ref{rem:explic}. Corresponding to the basis for $\mathbb{C}^{p+1}$ given by
\[(\eta_\infty',\eta_0',\eta_1',\ldots,\eta_{p-1}')= (\eta_\infty,\eta_0,\eta_1,\ldots,\eta_{p-1})(z_{i,j}),\]
we have a decomposition $\mathbb{C}^{p+1}=W\oplus W'$, where $W$ and $W'$ are the real part and imaginary part, respectively, and $\rho_0(\overline{\Gamma}_p)$ acts on each of them. Let $\hat{e}:|K|\rightarrow\mathbb{C}^{p+1}$ be the equivariant embedding in Proposition~\ref{pro:EEofSp}. Then, the projection $\mathbb{C}^{p+1}\rightarrow W$ induces a $\overline{\Gamma}_p$-equivariant map $\hat{e}_0:|K|\rightarrow W$ with respect to $\rho_0$. By checking those possible intersections of the triangles in $\hat{e}_0(|K|)$ with the triangle spanned by
\[\hat{e}_0([1:0]),\, \hat{e}_0([0:1]),\, \hat{e}_0([1:1]),\]
it is not hard to see that $\hat{e}_0$ is actually an embedding. Then we can also modify it to get a smooth embedding, as in the proof of Proposition~\ref{pro:EEofSp}.

\begin{proposition}
The map $\hat{e}_0:|K|\rightarrow W$ gives a $\overline{\Gamma}_p$-equivariant embedding, with respect to $\rho_0$. And there exists a smooth $\overline{\Gamma}_p$-equivariant embedding $e_0:\Sigma(p)\rightarrow W$, with respect to $\rho_0$, so that $e_0(\Sigma(p))$ lies near $\hat{e}_0(|K|)$.
\end{proposition}

We leave the details of the proof to the readers, and note that $(\hat{e}_0(|K|),\rho_0(\overline{\Gamma}_p))$ can be viewed as a generalization of the $\overline{\Gamma}_5$-action on a regular icosahedron.
\end{appendix}


\bibliographystyle{amsalpha}

\end{document}